\documentclass{article}
\usepackage{amssymb}
\usepackage{amsfonts}
\usepackage{amsmath}

\newtheorem{theorem}{Theorem}

\newtheorem{corollary}[theorem]{Corollary}
\newtheorem{criterion}[theorem]{Criterion}
\newtheorem{definition}[theorem]{Definition}
\newtheorem{example}[theorem]{Example}

\newtheorem{lemma}[theorem]{Lemma}

\newtheorem{proposition}[theorem]{Proposition}
\newtheorem{remark}[theorem]{Remark}

\newenvironment{proof}[1][Proof]{\noindent\textbf{#1.} }{\ \rule{0.5em}{0.5em}}

\newcommand{\Rea}{\mathop{\rm Re}\nolimits}

\begin{document}

\title{Growth Estimates for the Numerical Range of Holomorphic Mappings and
Applications}
\author{Filippo Bracci\footnote{Partially supported by the ERC grant ``HEVO -
Holomorphic Evolution Equations" no. 277691.}, Marina Levenshtein,
\\ Simeon Reich\footnote{Partially supported by the
 Israel Science Foundation (Grant No. 389/12), by the Fund for the
 Promotion of Research at the Technion and by the Technion General
 Research Fund} and David Shoikhet}

 \maketitle


\smallskip


\begin{abstract}
The numerical range of holomorphic mappings arises in many aspects
of nonlinear analysis, finite and infinite dimensional holomorphy,
and complex dynamical systems. In particular, this notion plays a
crucial role in establishing exponential and product formulas for
semigroups of holomorphic mappings, the study of flow invariance
and range conditions, geometric function theory in finite and
infinite dimensional Banach spaces, and in the study of complete
and semi-complete vector fields and their applications to starlike
and spirallike mappings, and to Bloch (univalence) radii for
locally biholomorphic mappings.

In the present paper we establish lower and upper bounds for the
numerical range of holomorphic mappings in Banach spaces. In
addition, we study and discuss some geometric and quantitative analytic
aspects of fixed point theory, nonlinear resolvents of holomorphic
mappings, Bloch radii, as well as radii of starlikeness and spirallikeness.
\end{abstract}

\section{Introduction and preliminaries}

Let $X^{\ast}$ denote the dual of a complex Banach space $X$ and
let $\langle x,x^{\ast}\rangle$ denote the duality pairing of
$x\in X$ and $x^{\ast}\in X^{\ast}$. For each $x\in X$, the set
$J(x)$, defined by
\begin{equation*}
J(x):=\{x^{\ast}\in X^{\ast }:\text{ }\langle x,x^{\ast}\rangle
=\|x\|^{2}=\|x^{\ast}\|^{2}\},
\end{equation*}
is not empty by the Hahn--Banach theorem, and is a closed and
convex subset of $X^{\ast}$.


Let $D$ be a domain in $X$ and let $f:D\rightarrow X$ be a
mapping. We use the notation $\sup_{x\in D} \Rea\langle
f(x),x^{\ast} \rangle$ for the supremum of $\Rea\langle f(x),
x^{\ast}\rangle$ over all pairs $x\in D$ and $x^{\ast}\in J(x)$.

We denote by $\mathcal{B}_{R}:=\{x\in X: \|x\|<R\}$ the open ball
centered at the origin of radius $R$ in the complex Banach space
$X$.

\begin{definition}[cf. \protect\cite{H} and \cite{HRS}]
\label{def1} Let $h:\overline{\mathcal{B}_{R}}\to X$  be
continuous on the closure $\overline{\mathcal{B}_{R}}$ of
$\mathcal{B}_{R}$. We define the set
\begin{equation*}
V_{\mathcal{B}_{R}}(h):=\{\langle h(x),x^{\ast}\rangle: \;
\|x\|=R, \; x^{\ast}\in J(x) \}
\end{equation*}
and call it {\sl the numerical range} of $h$ with respect to
$\mathcal{B}_{R}$.
\end{definition}

The number $|V_{\mathcal{B}_{R}}(h)|:=\sup \{|\langle
h(x),x^{\ast}\rangle|: \; \|x\|=R, \; x^{\ast}\in J(x)\}$ is
called the \textit{numerical radius} of $h$ with respect to
$\mathcal{B}_{R}$.

We denote by $\mathrm{Hol} (\mathcal{D}, X)$ the set of all
holomorphic mappings from a domain $\mathcal{D}\subset X$ into
$X$.

\begin{definition} [cf. \protect\cite{H} and \protect\cite{HRS}]
\label{def2a} Let $h\in \mathop{\rm Hol}\nolimits\nolimits(\mathcal{B}%
_{R},X) $. We say that $h$ is (holomorphically) dissipative if
\begin{equation*}
\limsup_{s \rightarrow 1^{-}}\sup \Rea \,
V_{\mathcal{B}_{R}}(h_{s})\leq 0,
\end{equation*}%
where $h_{s}(x) := h(sx)$, $0\leq s<1$.
\end{definition}

In view of their numerous applications, dissipative mappings which
are not necessarily holomorphic constitute an important class of
mappings in complex Banach spaces. In this paper we introduce the
following more general notion.

\begin{definition}
\label{disdf} Given $\omega , \theta \in \mathbb{R}$, a mapping
$h:\mathcal{B}_{R}\mathcal{\rightarrow }X$ is called $\left(
\omega ,\theta \right) $-dissipative (or just quasi-dissipative)
on $\mathcal{B} _{R} $ if there exists $\varepsilon
>0$ such that
\begin{equation}
\mathrm{Re}\, \langle e^{i\theta }h(x),x^{\ast }\rangle \leq
\omega ,\label{3b}
\end{equation}
for all $x$ satisfying $R-\varepsilon <\Vert x\Vert <R$ and
$x^{\ast}\in J(x)$.
\end{definition}

For a holomorphic mapping $h$ on the unit ball
$\mathcal{B}=\mathcal{B}_{1}$ the above definition means that for
each $s\in \left( 1-\varepsilon ,1\right) $, the closed convex
hull of the numerical range of each $h_{s}$ is not the whole
complex plane, or which is one and the same, $h:\mathcal{B}\to X$
is \textsl{\ quasi-dissipative} if there is $\varepsilon >0$ such
that the closed convex hull of the set
\begin{equation*}
\Omega _{\varepsilon }(h):=\left\{ \langle h(x),x^{\ast }\rangle
:\ 1-\varepsilon <\Vert x\Vert <1, \; x^{*}\in J(x)\right\}
\end{equation*}%
is not the whole complex plane $\mathbb{C}$.

Obviously, a holomorphic $\left( 0,0\right) $-dissipative mapping
on the unit ball is holomorphically dissipative in the sense of
Definition \ref{def2a}. Also note that $\left( 0,\pi \right)
$-dissipative mappings are sometimes called holomorphically
accretive \cite{ERS2000}.


One of the general problems we intend to study is the following
one:

\textit{Given a quasi-dissipative mapping }$h$ \textit{\ on the
open unit ball }$\mathcal{B}$, \textit{find }$r\in \left(
0,1\right)$ \textit{\ (if it exists) such that }$h$
\textit{\ is dissipative on the ball }$\mathcal{B} _{r}$\textit{.}

Since every holomorphic mapping $h$ on a domain $\mathcal{D}$ is locally
Lipschitzian, it follows that the Cauchy problem
\begin{equation*}
\begin{cases}
\displaystyle\frac{dx(t)}{dt}=h(x(t)) \\
x(0)=x_{0}%
\end{cases}%
\end{equation*}%
has a unique continuous solution $x\left( t\right) $ defined on the interval
$\left[ 0,T\right] ,$ where $T$ depends on the initial value $x_{0}\in
\mathcal{D}$.

The mapping $-h$ is said to be a \textit{semi-complete vector field
}on $ \mathcal{D}$ if for each $x_{0}\in \mathcal{D}$, this
solution is well defined on the right half-axis $\left[ 0,\infty
\right) $ and the values of $x\left( t\right) $ belong to
$\mathcal{D}$ for each initial data $ x_{0}\in \mathcal{D}$. In
this situation, $h$ generates a one-parameter semigroup of
holomorphic self-mappings of $\mathcal{D}$ \cite{ReichS}.

Note that it may happen that $-h$ is not semi-complete on the whole
domain of definition $\mathcal{D}$, but it is semi-complete on
some open subset of $ \mathcal{D}$. In this case we say that $-h$
is a \textit{locally semi-complete vector field.}

It is known (see, for example, \cite{ReichS}) that if $h$ is
holomorphic on $\mathcal{B}_{R}$, then $-h$ is semi-complete on
$\mathcal{B}_{R}$ if and only if $h$ is dissipative on
$\mathcal{B}_{R}.$

Moreover, it turns out that if the numerical range of $h$ (say,
with respect to the open unit ball $\mathcal{B}$ of $X$) is not the
whole complex
plane, then for each $r\in \left( 0,1\right) $, there is a real number $%
\omega =\omega \left( r\right) $ such that the perturbed mapping
$\omega I - h$ is semi-complete on the ball $\mathcal{B}_{r}$. The
question is how this number $\omega $ depends on $r$ and how to
find the minimal value of the function $\omega \left( r\right) $
with respect to $r\in \left( 0,1\right) .$

By using the exponential formula for semigroups of holomorphic mappings
(see, for example, \cite{ReichS}) one can see that this problem is
equivalent to the following one.\textit{\ Find a function }$\omega \left(
r\right) $\textit{\ on the interval }$\left( 0,1\right) $\textit{\ such that
the nonlinear resolvent }$\left( \lambda I-h\right) ^{-1}$\textit{\ is well
defined on the ball }$\mathcal{B}_{\rho }$\textit{\ of radius }$\rho =\left( \lambda
-\omega \left( r\right) \right) r$\textit{\ for all }$\lambda \geq \omega
\left( r\right) $\textit{\ and maps this ball into }$\mathcal{B}_{r}.$ We will study
this problem in more detail in the third section of our paper.

In this case it is also of interest, in analogy with the linear
theory, to determine if the associated resolvent mapping
$(\lambda I-h)^{-1}\circ (\lambda -\omega (r))I$ can be extended to a sector
in the complex plane with vertex at $\omega (r)$ and to
estimate
the angle of its opening.

For a holomorphic mapping $h:\mathcal{B}\to X$, one says that it
has {\sl unit radius of boundedness} if it is bounded on each
subset strictly inside $\mathcal{B}$ (\cite{FV, H}; see also
\cite{B-E-S2014, B-K-S2014}).

It follows from a result of L. A. Harris \cite{H}\ that a
holomorphic mapping on $\mathcal{B}$ has unit radius of
boundedness if and only if its numerical radius $\left\vert
V_{\mathcal{D}}(h)\right\vert $ is bounded with respect to any
convex subset $\mathcal{D}$ in $\mathcal{B}$. Moreover, it was
shown in \cite{HRS} that this is equivalent to a formally weaker
condition, namely,
\begin{equation*}
\limsup_{r \rightarrow 1^{-}}\sup_{\left\Vert x\right\Vert =r}
\Rea \, \left\langle h(x),x^{\ast }\right\rangle <\infty .
\end{equation*}

The problem of verifying whether a holomorphic mapping has unit
radius of boundedness, as well as the general study of its
numerical ranges, arise in many aspects of infinite dimensional
holomorphy (see, for example, \cite{FV, H}) and complex dynamical
systems \cite{A-R-S, ReichS}. In particular, they play a crucial
role in establishing exponential and product formulas for
semigroups of holomorphic mappings \cite{RS-SD-96, RS-97}, the
study of flow invariance and range conditions in nonlinear
analysis \cite{HRS, Reich}, and geometric function theory in finite
and infinite dimensional Banach spaces \cite{ReichS}. They were
specifically used for the class of
semi-complete vector fields (or infinitesimal generators) in
their applications to the study of starlike and spirallike
mappings \cite{ReichS}, and Bloch (univalence) radii \cite{HRS}
for locally biholomorphic mappings. Other pertinent papers include
\cite{G06}, \cite{G14} and \cite{RSZ}.

Observe also that the concept of unit radius of boundedness for
holomorphic mappings is a specific phenomenon in the infinite
dimensional case because in a finite dimensional Banach space each
holomorphic mapping on the unit ball is bounded on each subset
strictly inside the ball. This is no longer true in the general
case. Relevant examples can be found in \cite{SBCh}.

\section{Lower and upper bounds for the numerical range}

Let $\mathcal{D}$ be a domain in $X,$ $0\in \mathcal{D}$ and let
$h: \mathcal{D}\rightarrow X$ be holomorphic. Since $h$ is locally
bounded, there is a ball $\mathcal{B}_{R}$ in $\mathcal{D}$ such
that
\begin{equation}
N_{R}:=\sup_{\left\Vert x\right\Vert <R}\Rea \, \left\langle
h(x),x^{\ast }\right\rangle <\infty .  \label{NR}
\end{equation}
Also, for each $r< R$, we use the quantity
\begin{equation}
N_{r}:=\sup_{\left\Vert x\right\Vert =r}\Rea \, \left\langle
h(x),x^{\ast }\right\rangle  .\label{NR1}
\end{equation}

The first aim of this section is the following one.

\textit{(i)} \textit{Find an explicit upper bound} $\mathcal{F}(r)$
\textit{for} $N_{r}$ \textit{which depends only on} $h(0)$,
$h'(0)$ \textit{and} $N_{R}$ \textit{\and such that}
\begin{equation}
\lim_{r\rightarrow R^{-}}\mathcal{F}\left( r\right) =N_{R}.  \label{GELIM}
\end{equation}

On the other hand, it might happen that for a given $R>0$, the
value of $N_{R}\left( :=\sup_{\left\Vert x\right\Vert <R}
\Rea \, \left\langle h(x),x^{\ast }\right\rangle \right) $ is not bounded,
while the convex hull of the numerical range of $h$ is not the
whole plane. This is equivalent to the fact that for some
real $\theta $,
\begin{equation}
N_{R}\left( \theta \right) :=\sup_{\left\Vert x\right\Vert <R}
\Rea \, \left\langle e^{i\theta }h(x),x^{\ast }\right\rangle <\infty
\text{ .} \label{4.1}
\end{equation}%
In other words, it may happen that even if $N_{R}\left( \theta
\right) $ is finite for some real $\theta$, the value $N_{R}\left(
0\right) =$ $N_{R}$ \ is not finite. A simple example is given in
$X=\mathbb{C}$, the complex plane, by the mapping
$h(x)=ix\displaystyle\frac{1+x}{1-x}$ with $R=1$ and $\theta
=\displaystyle\frac{\pi}{2}$. Nevertheless, for each $r<R$, the
value $N_{r}$ is finite on the smaller ball $\{x\in X: \|x\|\leq
r\} $.

\text{ }
Therefore the following problem is also relevant.

\textit{(ii)} \textit{Knowing the value}  $N_{R}(\theta)$,
\textit{find an explicit upper bound }$\mathcal{F}\left( r,\theta
\right) $\textit{\ for }$\sup_{\left\Vert x\right\Vert =r}\Rea \,
\left\langle h(x),x^{\ast }\right\rangle $, \textit{which depends
only on}  $h(0)$, $h'(0)$, $\theta$ and $N_{R}(\theta)$, \textit{\
and such that} $\mathcal{F}\left( r,0\right) =\mathcal{F}\left(
r\right)$.

Then, by definition, this function also gives us growth estimates
for the upper bound of the numerical range of $h$ with respect to
$ r\leq R$ .

Similarly, one can consider the problem of finding growth
estimates for the function
\begin{equation*}
M_{r}(\theta ) := \sup_{\left\Vert x\right\Vert =r}\Rea \, \left\langle
e^{i\theta }(h(x)-h(0)),x^{\ast }\right\rangle
\end{equation*}%
by using a suitable computable function $\Psi \left( r,\theta \right) $
such that
\begin{equation*}
M_{r}(\theta )\leq \Psi \left( r,\theta \right) \quad \mbox{and}\quad
\lim_{r\rightarrow R^{-}}\Psi \left( r,0\right) =M_{R}\left(
0\right) ,
\end{equation*}%
where
\begin{equation*}
M_{R}(\theta ):=\sup_{\left\Vert x\right\Vert <R}\Rea \, \left\langle
e^{i\theta }(h(x)-h(0)),x^{\ast }\right\rangle . \quad
\end{equation*}

Let us start solving  Problem (i):

\begin{proposition}
\label{propT} Let $h$ be a holomorphic mapping on
$\mathcal{B}_{R}$ and assume that
\begin{equation*}
N_{R}=\sup_{\left\Vert x\right\Vert <R}\Rea \, \left\langle
h(x),x^{\ast }\right\rangle <\infty \text{ .}
\end{equation*}%
Then the following estimate holds:
\begin{equation*}
N_{r}=\sup_{\left\Vert x\right\Vert =r}\Rea \, \left\langle
h(x),x^{\ast }\right\rangle \leq \mathcal{F}\left( r\right)
\text{,} 
\end{equation*}%
where
\begin{equation*}
\mathcal{F}\left( r\right) :=r\left\Vert h(0)\right\Vert \left( 1-\frac{r^{2}
}{R^{2}}\right) +\frac{r^{2}}{R+r}\left[ \left( R-r\right) L+2r\cdot \frac{
N_{R}}{R^{2}}\right]
\end{equation*}
and $L:=\sup_{\left\Vert u\right\Vert =1}\Rea \, \left\langle
h^{\prime
}(0)u,u^{\ast }\right\rangle \leq \frac{N_{R}}{R^{2}}.$
\end{proposition}

\begin{proof}
Let $g: \Delta _{R}\rightarrow \mathbb{C}$ be a holomorphic
function in $\Delta _{R}:=\{z\in \mathbb{C}: |z|<R\}$ for some
$R>0$. By the Hadamard--Borel--Carath\'{e}odory inequality, for all $\zeta\in \Delta_R$ such that $|\zeta|=r\in (0,R)$ we have
\begin{equation}\label{ref2}
\Rea g(\zeta)\leq \frac{R-r}{R+r}\Rea
g(0)+\frac{2r}{R+r}\sup\limits_{|\xi|<R} \Rea g(\xi).
\end{equation}
Now, let $f:\Delta_R\to \mathbb C$ be a holomorphic function and consider the holomorphic function $g:\Delta_R\to \mathbb C$ defined by
\begin{equation}\label{gfun}
g(\zeta):= \begin{cases}
\frac{f(\zeta)-f(0)}{\zeta}+\frac{\overline{f(0)}}{R^{2}}\zeta, &\quad \zeta\neq 0\\
f'(0)&\quad \zeta=0.
\end{cases}
\end{equation}
Note that 
\begin{equation}\label{boundprime}
\begin{split}
\Rea f'(0)&=\Rea g(0)\leq \limsup_{r\to R^-} \max_{|\zeta|=r}\Rea g(\zeta)\\&=
\limsup_{r\to R^-} \max_{|\zeta|=r}\frac{1}{r^{2}}\left[ \Rea  f(\zeta )\overline{\zeta }-
\Rea  \left( f(0)\overline{\zeta }\left(
1-\frac{r^{2}}{R^{2}}\right) \right) \right]\\
&\leq \frac{1}{R^2} \sup_{|\zeta|<R}\Rea f(\zeta)\overline{\zeta}.
\end{split}
\end{equation}

Moreover, applying \eqref{ref2} to \eqref{gfun}, 
 for all $\zeta$ such that $|\zeta|=r\in (0,R)$, we obtain
\begin{equation}\label{ref1}
\begin{split}
\Rea (\overline{\zeta}f(\zeta))&\leq \left(\! 1-\frac{r^{2}}{R^{2}}\!
\right)\Rea (\overline{\zeta}f(0))\\&+r^{2}\! \left[\! \Rea
f'(0)\frac{R-r}{R+r}+\frac{2r}{R^{2}(R+r)}\sup\limits _{|\xi|<R}
\Rea \left(\overline{\xi}f(\xi)\right)\! \right]\! .
\end{split}
\end{equation}

Let now $u\in X$, $\|u\|=1$. Fix  $u^{*}\in J(u)$ and consider the
holomorphic function $f:\Delta _{R}\rightarrow \mathbb{C}
$ defined by $f(\zeta )=\left\langle h(\zeta u),u^{\ast }\right\rangle$. 
Applying \eqref{ref1} to this $f$ and taking into account \eqref{boundprime}, we can easily complete the proof of the  proposition.
\end{proof}

\begin{remark}
Note that inequality \eqref{ref1} is, in fact, equivalent to the Hadamard--Borel--Carath\'{e}odory inequality \eqref{ref2}. Simply apply \eqref{ref1} to the holomorphic function $f(\zeta)=\zeta g(\zeta)$.
\end{remark}

The solution to Problem (ii) is the content of the next result:

\begin{proposition}
\label{propN} Let $h$ be a holomorphic mapping on
$\mathcal{B}_{R}$ and assume that
\begin{equation*}
N_{R}\left( \theta \right) :=\sup_{\left\Vert x\right\Vert <R}\Rea \,
\left\langle e^{i\theta }h(x),x^{\ast }\right\rangle <\infty .
\end{equation*}%
Then
\begin{equation*}
\sup_{\left\Vert x\right\Vert =r}\Rea \, \left\langle h(x),x^{\ast
}\right\rangle \leq \mathcal{F}_{1}\left( r,\theta \right) ,
\end{equation*}%
where
\begin{equation*}
\begin{split}
\mathcal{F}_{1}\left( r,\theta \right) :=& \sup_{\left\Vert x\right\Vert =r}%
\Rea \, \left( \left\langle h(0),x^{\ast }\right\rangle -\frac{r^{2}}{R^{2}}%
\overline{\left\langle e^{i\theta }h(0),x^{\ast }\right\rangle }\right) \\
& +r^{2}\left[ L+ \mathcal{L}\left( \theta
,r\right)\left(\frac{1}{R^{2}}N_{R}\left( \theta \right)-l(\theta
)\right)  \right]
\end{split}%
\end{equation*}%
with 
\begin{equation}\label{b00} \mathcal{L} ( \theta
,r) =\frac{2r
( R-r \cos \theta) }{
R^{2}-r^{2}}, \quad 
l\left( \theta \right) := \inf_{\left\Vert u\right\Vert =1}\Rea \,
\left\langle e^{i\theta }h^{\prime }(0)u,u^{\ast }\right\rangle,
\end{equation}
and $L:=\sup_{\left\Vert u\right\Vert =1}\Rea \, \left\langle
h^{\prime
}(0)u,u^{\ast }\right\rangle \leq \frac{N_{R}}{R^{2}}.$ 
\end{proposition}

\begin{proof}
Let $g:\Delta_R\to \mathbb C$ be a holomorphic function. By the Kresin--Maz'ya inequalities (see \cite{K-M}), for all $\zeta\in \Delta_R$ with $|\zeta|=r<R$, we have
\begin{equation}\label{mazja}
 \Rea  \left(
e^{i\theta }\left( g(\zeta )-g(0)
\right) \right) \\  \leq \mathcal{L}( \theta ,r)
[ \sup_{\left\vert \xi \right\vert <R}\Rea g(\xi )-\Rea g(0)].
\end{equation}
Now, given a holomorphic function $f:\Delta_R\to \mathbb C$, let $g$ be the holomorphic function  defined in \eqref{gfun}. By   \eqref{boundprime}, 
\[
\limsup_{r\to R^-} \max_{|\zeta|=r}\Rea g(\zeta)\leq \frac{1}{R^2} \sup_{|\zeta|<R}\Rea f(\zeta)\overline{\zeta}.
\]
Hence,  applying \eqref{mazja} to the function $g$,   for all $\zeta$ such that $|\zeta|=r\in (0,R)$, we obtain
\begin{equation}\label{mazyac}
\begin{split}
\Rea  e^{i\theta }&  f(\zeta)\overline{\zeta} \leq \Rea \, e^{i\theta }\left(
f(0)
-\frac{1}{R^{2}}\overline{f(0)
}r^{2}\right) \\ & +r^{2}\left[ \Rea \left(  f^{\prime
}(0) \left( e^{i\theta
}-\mathcal{L}( \theta ,r\right)) \right)
+\frac{1}{R^{2}}\mathcal{L}\left( \theta ,r\right) \sup_{|\xi|<R}\Rea f(\xi)\overline{\xi}\right].
\end{split}%
\end{equation}%
Let now $u\in X$, $\|u\|=1$. Fix  $u^{*}\in J(u)$ and consider the
holomorphic function $f:\Delta _{R}\rightarrow \mathbb{C}
$ defined by $f(\zeta )=\left\langle h(\zeta u),u^{\ast }\right\rangle$. 
Applying \eqref{mazyac} to such an $f$ and taking into account \eqref{boundprime},  we can easily finish the proof of this proposition.
\end{proof}

\begin{remark}
The proofs of Propositions \ref{propT} and \ref{propN} we gave in the original manu\-script were more involved and did not rely immediately on the Hadamard--Borel--Cara\-th\'{e}odory--Kresin--Maz'ya inequalities. We thank one of the anonymous referees for suggesting the shorter proofs contained here.
\end{remark}

\begin{remark}\label{lem}
Let $h$ be a holomorphic mapping on
$\mathcal{B}_{R}$ and assume that $N_R<+\infty$. Suppose that $h(0)=0$. From the proof of Proposition \ref{propN}, using \eqref{mazyac} with $\theta=0$, it follows that for all $x\in \mathcal B_R$ and $x^\ast\in J(x)$,
\begin{equation*}
\Rea \langle h(x), x^\ast\rangle \leq \Rea \left[ \langle h'(0)x, x^\ast\rangle (1-\mathcal L (0, r))\right]+\frac{\|x\|^2}{R^2} \mathcal L(0,r) N_R.
\end{equation*}
\end{remark}

 As we have already mentioned, sometimes it is more
convenient to study growth estimates for the function
\begin{equation*}
M_{r}(\theta )=\sup_{\left\Vert x\right\Vert =r}\Rea \,
\left\langle e^{i\theta }(h(x)-h(0)),x^{\ast }\right\rangle
\end{equation*}
with respect to $r\in \left( 0,R\right) .$ Of course, this can be
done by using Proposition \ref{propN}. However, in the same spirit of the previous proof, but using a different auxiliary function $f$, one can obtain slightly better estimates. As a matter of notation, if $h:\mathcal{B}_{R}\to X$ is holomorphic, $r\in (0,R)$ and $\theta\in \mathbb R$, let
\[
L(\theta ) := \sup_{\left\Vert u\right\Vert =1}\Rea \, \left\langle
e^{i\theta }h^{\prime }(0)u,u^{\ast }\right\rangle,
\]
and
\begin{equation*}
m_{r}(\theta ):=\inf_{\left\Vert x\right\Vert =r}\Rea \, \left\langle
e^{i\theta }\left( h(x)-h\left( 0\right) \right) ,x^{\ast
}\right\rangle.
\end{equation*}%
Moreover, $M_R(\theta):=\sup_{\left\Vert x\right\Vert <R}\Rea \,
\left\langle e^{i\theta }(h(x)-h(0)),x^{\ast }\right\rangle$, and similarly we define $m_R(\theta)$. The symbols $\mathcal{L} (\theta,r )$ and $l(\theta)$ were introduced in \eqref{b00}. With these notations at hand we can state and prove the following result:

\begin{proposition}
\label{prop1} Let $h$ be a holomorphic mapping on the ball $\mathcal{B}_{R}$
in $X.$ Given $\theta \in\mathbb{R}$, $R>0$ and $r\in (0,R]$, we have
\begin{equation}
m_{r}(\theta )\leq r^{2}l(\theta )\leq r^{2}L(\theta )\leq M_{r}(\theta).
\label{4.60}
\end{equation}
Moreover, for all $x\in \mathcal{B}_{R}$ such that $\left\Vert x\right\Vert =r<R$ and $x^{\ast}\in J(x)$ we have
\begin{equation}  \label{4.6a}
\begin{split}
r^{2}&\left( l(0)+\mathcal{L} (\theta,r )\left( \frac{m_{R}(\theta
)}{R^{2}} -L(\theta )\right) \right) \leq \Rea \, \left\langle
\left( h(x)-h(0)\right) ,x^{\ast }\right\rangle \\ & \leq
r^{2}\left[ \mathcal{L} (\theta,r )\left( \frac{M_{R}(\theta
)}{R^{2}} -l(\theta )\right) +L(0)\right].
\end{split}%
\end{equation}
\end{proposition}

\begin{proof}
The proof is a slight modification of the proof of Proposition \ref{propN}, plus the lower bound estimates given by the Kresin--Maz'ya inequalities. 

If $f:\Delta_R\to \mathbb C$ is a holomorphic function, we consider the auxiliary holomorphic function $g: \Delta_R\to \mathbb C$ defined by 
\begin{equation}\label{gfun2}
g(\zeta):= \begin{cases}
e^{i\theta}\frac{f(\zeta)-f(0)}{\zeta} &\quad \zeta\neq 0,\\
e^{i\theta}f'(0)&\quad \zeta=0.
\end{cases}
\end{equation}
Applying \eqref{mazja} with $-\theta$ instead of $\theta$ to the function $g$, and taking into account that $\mathcal L(\theta,r)=\mathcal L(-\theta, r)$, 
 for $\zeta\in \Delta_R$, $|\zeta| =r<R$, we obtain
\begin{equation}
\begin{split}
\Rea[ f(\zeta)-f(0)]\overline{\zeta} \leq & r^{2}\bigl(%
\mathcal{L} ( \theta,r ) \frac{1}{R^{2}}\sup_{|\xi| <R}\Rea \left(e^{i\theta}[f(\zeta)-f(0)]\overline{\zeta} \right)\\ &  +\Rea  \left[ (e^{-i\theta
}-\mathcal{L} \left( \theta,r \right) )e^{i\theta
}f'(0) \right] \bigr) .
\label{b02}
\end{split}%
\end{equation}
Let now $u\in X$, $\|u\|=1$. Fix  $u^{*}\in J(u)$ and consider the
holomorphic function $f:\Delta _{R}\rightarrow \mathbb{C}
$ defined by $f(\zeta )=\left\langle h(\zeta u),u^{\ast }\right\rangle$.  Applying \eqref{b02} to $f$, we obtain
\begin{equation}
M_{r}(0)\leq r^{2}\left[ \mathcal{L} (\theta,r )\left( \frac{M_{R}(\theta )}{%
R^{2}}-l(\theta )\right) +L(0)\right], 
\end{equation}
which gives the upper bound in \eqref{4.6a}.

Note that for each $\theta \in \mathbb{R}$ and $x\in
\mathcal{B}_{R}$ such that $\left\Vert x\right\Vert =r\leq R$,
$x=\zeta u$, $\left\vert \zeta \right\vert =r$, $x^{\ast}\in
J(x)$, it follows from the maximum principle for harmonic
functions
that%
\begin{equation*}
\begin{split}
\Rea \, \left\langle e^{i\theta }h^{\prime }(0)u,u^{\ast
}\right\rangle &=\Rea e^{i\theta}f'(0)=
\Rea \, g(0)\leq \max_{\left\vert \zeta \right\vert =r}\Rea \, g(\zeta ) \\ &\leq \frac{1}{r^{2}}\sup_{\left\Vert
x\right\Vert =r}\Rea \, \left\langle e^{i\theta }\left[
h(x)-h(0)\right] ,x^{\ast }\right\rangle =\frac{1}{r^{2}}
M_{r}(\theta ),
\end{split}%
\end{equation*}
which implies that, for each $r\in \left( 0,R\right) $,
\begin{equation*}
l(\theta ) \leq L(\theta ) \leq \frac{M_{r}(\theta )}{r^{2}},
\end{equation*}
giving the upper estimates in \eqref{4.60}.

In order to get the lower estimates, let us recall the  Kresin--Maz'ya lower bound (see \cite{K-M}). Let $g:\Delta_R\to \mathbb C$ be a holomorphic function. Then, given $\theta\in \mathbb R$, for all $\zeta\in \Delta_R$ with $|\zeta|=r<R$ we have
\begin{equation}\label{mazjainf}
 \Rea  \left(
e^{i\theta }\left( g(\zeta )-g(0)
\right) \right) \\  \geq \mathcal{L}( \theta ,r)
[ \inf_{\left\vert \xi \right\vert <R}\Rea g(\xi )-\Rea g(0)].
\end{equation}
Then, one can argue exactly as before, just replacing   \eqref{mazja} with  \eqref{mazjainf}. \end{proof}

\begin{corollary}[Rigidity property]
\label{cor(rigidity)} Let $h:\mathcal{B} _{R}\rightarrow X$ be holomorphic
and assume that for some $\theta \in \mathbb{R}$, one of the following
equalities holds:
\begin{equation*}
M_{R}(\theta )=R^{2}l\left( \theta \right)
\end{equation*}%
or%
\begin{equation*}
m_{R}(\theta )=R^{2}L\left( \theta \right) .
\end{equation*}%
Then the second equality holds too and $h$ is, in fact, an affine mapping: $%
h\left( x\right) =h^{\prime }\left( 0\right) x+h\left( 0\right) .$ In
particular, if $h\left( 0\right) =0,$ then $h$ is a linear operator, that is, $%
h\left( x\right) =h^{\prime }\left( 0\right) x$.
\end{corollary}

\begin{proof}
If the first equality holds, then $l\left( \theta\right) =L\left(
\theta\right)$ by (\ref{4.60}), and hence by (\ref{4.6a}),
$l\left( 0\right) =L\left( 0\right) ,$ whence by (\ref{4.6a}),
\begin{equation*}
\Rea \, \left\langle \left( h(x)-h(0)\right) ,x^{\ast
}\right\rangle \leq r^{2}l\left( 0\right)=\inf_{\left\Vert
u\right\Vert =1}\Rea \, \left\langle h^{\prime }(0)u,u^{\ast
}\right\rangle ,
\end{equation*}%
where $\|x\|=r<R$.

Since $\langle h'(0)u,u^{\ast}\rangle =g(0)$, where, as above,
$g(\zeta)=\frac{1}{\zeta}\langle h(\zeta u)-h(0), u^{\ast}
\rangle$, this implies that $$\Rea g(\zeta)\leq \Rea g(0), \quad
|\zeta |<R.$$ Hence, by the maximum principle for harmonic
functions, $\Rea g(\zeta)=\Rea g(0)$ in $\Delta _{R}$, which means
that $$\langle h(x)-h(0)-h'(0)x,x^{\ast}\rangle =0$$ for all $x\in
$ $\mathcal{B}_{R}$ and $x^{\ast}\in J(x)$.

Now it follows from Proposition 1 in \cite{H} that $h\left( x\right) =h\left(
0\right) +h^{\prime }\left( 0\right) x$ and we are done.

Similarly, one can get the same conclusion if the second equality
holds.
\end{proof}

\begin{remark}
Let $h: \mathcal B_R\to X$ be holomorphic. Let $M_{R}=M_{R}(0)$ and $u\in X$, $\|u\|=1$. Fix  $u^{*}\in J(u)$ and consider the
holomorphic function $f:\Delta _{R}\rightarrow \mathbb{C}
$ defined by $f(\zeta )=\left\langle h(\zeta u),u^{\ast }\right\rangle$. Equation \eqref{ref2} applied to the function $g$ defined in \eqref{gfun} implies that
\begin{equation}
r^{2}L\leq M_{r}\leq \frac{R-r}{R+r}L+\frac{2r}{R+r}M_{R}.  \label{M12}
\end{equation}%
Also, by the same token and using the classical Littlewood
two-sided estimates (see \cite{K-M}), one can establish another lower bound for
$M_{r}.$ Namely,
\begin{equation*}
r^{2}p\left( r\right) \leq M_{r},
\end{equation*}%
where
\begin{equation*}
p\left( r\right) =\frac{R+r}{R-r}\cdot L-\frac{2r}{R-r}\cdot \frac{1}{R^{2}}%
M_{R}.
\end{equation*}%
However, the left-hand side inequality in (\ref{M12}) is better. Indeed,
this inequality says that for all $r\in (0,R],$
$ L\leq \frac{M_{r}}{r^{2}}$,
and, in particular,%
\begin{equation*}
L\leq \frac{M_{R}}{R^{2}},
\end{equation*}%
which, in turn, implies that for $r\in (0,R)$,
\begin{equation*}
p\left( r\right) \leq \frac{R+r}{R-r}\cdot L-\frac{2r}{R-r}\cdot L=L.
\end{equation*}
\end{remark}

An immediate consequence of Proposition \ref{propT} is the following
growth estimate for the numerical radius which, in its turn implies, by Proposition 1 in
\cite{H},  an estimate for the growth of the norm of $h$.

\begin{corollary}
\label{num.rad} Let $h:\mathcal{B}_{R}\to X$ be a holomorphic
mapping on $\mathcal{B}_{R}$ with $h\left( 0\right) =0$ and assume
that $N_{R}\left( h\right) :=\sup_{\left\Vert x\right\Vert <R}\Rea \,
\left\langle h(x),x^{\ast }\right\rangle $ is finite. Then for
each $r\in \left( 0,R\right) $, the values $\left\vert
V_{r}(h)\right\vert =\sup_{\left\Vert x\right\Vert <r}\left\vert
\left\langle h(x),x^{\ast }\right\rangle \right\vert $ and
$W_{r}\left( h\right) =\sup_{\left\Vert x\right\Vert \leq
r}\left\Vert h(x)\right\Vert $ are finite. Moreover,
\begin{equation*}
\left\vert V_{r}(h)\right\vert \leq \frac{r^2}{R+r}\left[(R-r)L+\frac{2rN_R}{R^2} \right],
\end{equation*}%
and
\begin{equation*}
\left| W_{r}( h)\right| \leq \frac{2R^{2}}{\left( R-r\right)
^{2}}\left\vert V_{R}(h)\right\vert.
\end{equation*}%
\end{corollary}

In particular, if the closed convex hull of the numerical range of a
holomorphic mapping on the unit ball is not the whole complex plane, then it
has unit radius of boundedness.

\section{Nonlinear resolvents of holomorphic mappings and semi-complete
vector fields}

We start this section with the following notions.

\begin{definition}
\label{def1.2} Let $\mathcal{D}$ be a domain in $X$,
$\mathcal{D}\ni 0$, and let $h:\mathcal{D}\to X$ be holomorphic.
We define the \textbf{resolvent set} $\rho(h) \subseteq \mathbb{C}$ of $h$ to be
the set of those complex numbers $\lambda \in \mathbb{C}$ for which
there is an open set $\mathcal{D} _{\lambda }\subseteq
\mathcal{D}$, $\mathcal{D}_{\lambda }\ni 0$, such that $ \lambda
I-h$ is holomorphically invertible on $\mathcal{D}_{\lambda }$.
The complement $\sigma (h)$ of $\rho (h)$ is called the
\textbf{spectrum} of $h$.
\end{definition}

In other words, the spectrum $\sigma (h)$ of $h$ consists of those
$\lambda \in \mathbb{C} $ such that it is not possible to find an
open subset $\mathcal{D} _{\lambda }$ and a neighborhood
$V_{\lambda }\subseteq (\lambda I-h)\mathcal{ D}_{\lambda }$,
$V_{\lambda }\ni -h(0)$, such that $(\lambda I-h)^{-1}$
is a well-defined holomorphic mapping on $V_{\lambda }$ with values in
$\mathcal{ D}_{\lambda }$.

\begin{remark}
\label{rem.A} It was shown by L. A. Harris that $\sigma (h)=\sigma
(h^{\prime }(0))$ and, respectively, $\rho (h)=\rho (h^{\prime
}(0))$ (see \cite{H}).
\end{remark}

We set $\Re (\lambda ,h):= (\lambda I-h)^{-1}$ whenever it exists on
an open domain $V_{\lambda }\, (\ni -h(0))$, and $V_{\lambda }$ is
called the domain of the resolvent $\Re (\lambda ,h)$.

We will see below that the properties of the resolvent set of a
holomorphic mapping as well as the domain of definition of its resolvent  can be described in terms of the numerical range of the given mapping.

As we have already mentioned in Section 1, the problem of finding a
domain $ V_{\lambda }$ for the existence of the resolvent $\Re
(\lambda ,h)$ is related to the problem of local and global
descriptions of semi-complete vector fields. This observation is
based on the following fact (see, for example, \cite{ReichS}).

\begin{criterion}
Let $\mathcal{D}$ be a bounded and convex domain in a complex Banach
space $X$, and let $h:\mathcal{D}\to X$ be a holomorphic mapping
on $\mathcal{D}$. Then for some real number $\mu $, the mapping
$\mu I -h$ is a semi-complete vector field on $\mathcal{D}$ if and
only if the equation
\begin{equation*}
(\lambda I-h)(x)=(\lambda -\mu )y
\end{equation*}%
has a unique solution $x=\Re (\lambda ,h)\circ ((\lambda -\mu
)I)\left( y\right) $ for each $y\in \mathcal{D}$.
\end{criterion}
We call the mapping $\Phi _{\lambda }:=\Re (\lambda ,h)\circ
((\lambda -\mu )I): \mathcal{D\rightarrow D}$
the {\it associated resolvent mapping} of $h-\mu I.$

As above, let $h$ be a holomorphic mapping on the ball
$\mathcal{B}_{R}$ and assume that
\begin{equation*}
N_{R}=\sup_{\left\Vert x\right\Vert <R}\Rea \, \left\langle
h(x),x^{\ast }\right\rangle <\infty \text{ .}
\end{equation*}

Consider the resolvent equation
\begin{equation}
\lambda x-h(x)=z\text{, } \quad z\in X .  \label{4.12}
\end{equation}

For a fixed $r\in \left( 0,R\right) $, we would like to find conditions
which ensure that (\ref{4.12}) has a unique solution $x=x(z)\in \mathcal{B}%
_{r}$. To this end, we define the mapping $G:\mathcal{B}_{r}\to X$
by the formula%
\begin{equation*}
G(x):=z-\lambda x+h(x)\text{.}
\end{equation*}%
Then for every $x\in \partial \mathcal{B}_{r}$ and $x^{\ast}\in
J(x)$, $\Rea \left\langle G(x),x^{\ast }\right\rangle \leq
\left\Vert z\right\Vert r-r^{2}\Rea \lambda +r\omega (r)$, where,
by Proposition \ref{propT},
\begin{equation}
\omega \left( r\right) =\frac{1}{r}\mathcal{F}\left( r\right) =\left\Vert
h(0)\right\Vert \left( 1-\frac{r^{2}}{R^{2}}\right) +\frac{r}{R+r}\left[
\left( R-r\right) L+2r\cdot \frac{N_{R}}{R^{2}}\right]  \label{W}
\end{equation}%
with $L=\sup_{\left\Vert u\right\Vert =1}\Rea \left\langle
h^{\prime }(0)u,u^{\ast }\right\rangle .$

Hence, if
\begin{equation}
\left\Vert z\right\Vert +\omega (r)<r\Rea \lambda, \text{ }
\label{4.13}
\end{equation}
then we obtain the inequality
\begin{equation*}
\sup_{\left\Vert x\right\Vert =r}\Rea \left\langle G(x),x^{\ast
}\right\rangle <0 ,
\end{equation*}%
which implies the existence and the uniqueness of solutions to
(\ref{4.12}) (see \cite{HRS}).

Let now $\mu=\mu (r)$, $r\in [0,R]$, be a finite real-valued function and
assume that $\lambda \in \mathbb{R}$, $\lambda >\mu (r)$.

Solving (\ref{4.13}) with $z=(\lambda -\mu )y,$ where $y\in \mathcal{B}_{r},$
we get the condition
\begin{equation}
\lambda >\max \{\eta (r),\mu (r)\}\text{ ,}  \label{4.14}
\end{equation}%
where
\begin{equation}
\eta (r)=\frac{\omega (r)-\mu (r)\left\Vert y\right\Vert }{r-\left\Vert
y\right\Vert }.  \label{4.15}
\end{equation}
Note that $\mu (r)\geq \eta (r)$ if and only if $\mu (r)\cdot r\geq \omega
(r)$. In particular, if $\mu (r)\cdot r=\omega (r)$, then
\begin{equation*}
\mu (r)=\eta (r).
\end{equation*}
Finally, one can use classical calculus to investigate the function $\mu
(r)=\frac{1}{r}\omega (r)$.

Denote $b:=\displaystyle\frac{N_{R}}{R^{2}}$, $c:=\Vert h(0)\Vert $, and
consider the real-valued function
\begin{equation}
\mu (r)=\frac{R-r}{R+r}L+\frac{2r}{R+r}b+\frac{c}{r}\left( 1-\frac{r^{2}}{%
R^{2}}\right) ,\quad r\in (0,R).
\end{equation}%
If $c=0$, then $\mu (r)$ has no strict minimum in $(0,R)$. Indeed, in the case $c=0$ and
$b=L$, $\mu $ is the constant function $\mu (r)=L$,  and if $c=0$ and $%
b\neq L$, the derivative $\mu ^{\prime }(r)=\frac{2R(b-L)}{(R+r)^{2}}>0$, so
$\mu (r)$ is strictly increasing with $\mu (0)=L$ and $\mu (R)=b$.

Otherwise, if $c\neq 0,$ then the derivative
\begin{equation*}
\mu ^{\prime }(r)\!=\!\frac{2R(b-L)}{(R+r)^{2}}\!-\!c\frac{r^{2}+R^{2}}{%
R^{2}r^{2}}\!=\!-\!\frac{c\left( R^{2}\!+\!r^{2}\right) ^{2}\!+\!2crR\left(
R^{2}+r^{2}\right) -2R(b-L)(Rr)^{2}}{(R+r)^{2}R^{2}r^{2}}
\end{equation*}%
equals zero if and only if $R^{2}+r^{2}=\beta Rr$, where $\beta =%
\displaystyle\frac{\sqrt{c^{2}+2R(b-L)c}-c}{c}$.

\text{}

If $\beta \leq 2$, then the
equation $R^{2}+r^{2}=\beta Rr$ has no solution in $(0,R)$. If $\beta >2$
or, which is one and the same, $R(b-L)>4c$,
then it has the unique solution $r_{\ast }=R%
\displaystyle\frac{\beta -\sqrt{\beta ^{2}-4}}{2}$ in $(0,R)$ and since $\mu
^{\prime }(R)>0$, $\mu $ has a minimum at $r_{\ast }$.

Thus we have arrived at the following result.

\begin{proposition}
\label{teor4.1} Let $h:\mathcal{B}_{R}\to X$ be holomorphic in $\mathcal{%
B}_{R}$ and let
\begin{equation*}
N_{R}=\sup_{\left\Vert x\right\Vert =R}\Rea  \left\langle
h(x),x^{\ast }\right\rangle <\infty .
\end{equation*}
The following assertions hold:

(i) If $\omega (r)$ is given by (\ref{W}), then for each $z\in X$ and all $%
\lambda \in \mathbb{C}$ such that
\begin{equation*}
\Rea \lambda >\frac{1}{r}\left[ \omega (r)+\left\Vert z\right\Vert %
\right],
\end{equation*}%
the equation
\begin{equation*}
(\lambda I-h)(x)=z
\end{equation*}%
has a unique solution $x=\Re (\lambda ,h)\left( z\right) \in
\mathcal{B}_{r}$.

(ii) In particular, if we define $\mu \, (=\mu (r))=\frac{1}{r} \omega (r)$,
then for each $\lambda >\mu $, the equation
\begin{equation*}
(\lambda I-h)(x)=(\lambda -\mu )y
\end{equation*}%
has a unique solution $x=x(y)\in \mathcal{B}_{r}$ whenever $y\in \mathcal{B}
_{r}$.

(iii) If $h\left( 0\right) =0$, then the function $\mu =\mu (r)$
has no strict minimum in the interval $\left( 0,R\right)$. In
particular, it is either the constant function $ \mu (r)=L$ if $N_{R}=R^{2}L$ or
a strictly increasing function otherwise.

(iv) If $h\left( 0\right) \neq 0,$ then the function $\mu =\mu
(r)$ has a minimum at the point $r_{\ast
}=R\displaystyle\frac{\beta -\sqrt{\beta ^{2}-4 }}{2}$ in $(0,R),$
where $\beta =\displaystyle\frac{\sqrt{c^{2}+2R(b-L)c}-c}{ c},$ if
and only if
\begin{equation}
R\left(\frac{N_{R}}{R^{2}}-L\right)>4\left\Vert h(0)\right\Vert .
\label{Suf}
\end{equation}
\end{proposition}

We recall that a mapping $-h:\mathcal{B}_{R}\to X$ is said to be
\textit{a locally semi-complete vector field }if there is $r\in
\left( 0,R\right) $ such that $-h$ is semi-complete on
$\mathcal{B}_{r}.$ It follows from the above criterion  that $-h$ is
a locally semi-complete vector field whenever the function $\mu
(r)$ vanishes or is negative in the interval $\left( 0,R\right).$

The simplest situation occurs when $h\left( 0\right) =0.$ Note that
a necessary condition for $\mu (r)$ to vanish at some
point of the interval $\left( 0,R\right)$ is that $L<0.$

Thus we get the following conclusion.

\begin{corollary}
\label{semi} Let $h:\mathcal{B}_{R}\to X$ be holomorphic in $\mathcal{B}%
_{R}$ with $h\left( 0\right) =0,$ and let
\begin{equation*}
N_{R}=\sup_{\left\Vert x\right\Vert <R}\Rea \left\langle
h(x),x^{\ast }\right\rangle <\infty .
\end{equation*}%
Then $-h$ is a locally semi-complete vector field if and only if
the following condition holds:
\begin{equation*}
L=\sup_{\left\Vert u\right\Vert =1}\Rea \left\langle h^{\prime
}(0)u,u^{\ast }\right\rangle <\min \left\{
0,\frac{N_{R}}{R^{2}}\right\} .
\end{equation*}
In this case $-h$ is semi-complete on each $\mathcal{B}_{r}$ with
$r\in \left( 0,\displaystyle\frac{-R^{3}L}{2N_{R}-LR^{2}}\right)
$.
\end{corollary}

Since in general when $h\left( 0\right) \neq 0$ the equation $\mu
(r)=0$ is equivalent to a third order algebraic equation it can be
seen by using Vieta's formulas that under condition (\ref{Suf})
this equation has three positive roots $\left\{ r_{i}\right\}
_{i=1}^{3}$ such that $0<r_{1}<r_{\ast }<r_{2}\leq R<r_{3}$ if and
only if $\mu (r_{\ast })<0.$

\begin{theorem}
\label{cors}Let $h:\mathcal{B}_{R}\to X$ be holomorphic in $\mathcal{B}%
_{R}$ with $h\left( 0\right) \neq 0,$ and let
\begin{equation*}
N_{R}=\sup_{\left\Vert x\right\Vert <R}\Rea \left\langle
h(x),x^{\ast }\right\rangle <\infty .
\end{equation*}

\noindent If $R\left( \displaystyle\frac{N_{R}}{R^{2}}-L\right) >4\Vert
h(0)\Vert $ and $\mu (r_{\ast })<0,$
where $r_{\ast}=R\displaystyle\frac{\beta -\sqrt{\beta
^{2}-4}}{2}\in (0,R),$ with $\beta =\displaystyle\frac{\sqrt{c^{2}+2R(b-L)c%
}-c}{c},$ then $-h$ is semi-complete on each $\mathcal{B}_{r}$
with $r \in \left( r_{1},r_{2}\right) $, where $r_{1}<r_{2}$ are the
roots of the equation $\mu (r)=0$ in $(0,R]$.
\end{theorem}

\begin{remark}\label{roots}
One can find the values of $r_{1}$ and $r_{2}$ by the formulas
$$r_{1}=-2\sqrt{Q}\cos \phi -\frac{m}{3}, \quad
r_{2}=-2\sqrt{Q}\cos \left(\phi -\frac{2\pi}{3}\right)
-\frac{m}{3},$$ where $m=\displaystyle\frac{R(RL-2Rb+c)}{c}$,
$\phi =\displaystyle\frac{1}{3}\arccos \frac{A}{\sqrt{Q^{3}}}$,
$Q=\displaystyle\frac{cm^{2}+3LR^{3}+3cR^{2}}{9c}$ and
$A=\displaystyle\frac{2cm^{3}+9m(LR^{3}+cR^{2})-27cR^{3}}{54c}$.
\end{remark}

The situation becomes more transparent if we assume that $
N_{R}=0. $ In this case one of the roots is equal to $R$ and,
actually, condition (\ref{Suf}) already yields
$\mu (r_{\ast})<0$; hence $ 0<r_{1}<$ $r_{2}=R.$
Since this situation is of some special interest in the open unit ball
$\mathcal{B}$ and has applications to fixed point theory, we will describe
it separately.

\begin{corollary}
\label{nullp}Let $-h$ be a semi-complete vector field on the open
unit ball $ \mathcal{B}$ in $X.$ Then $L=\sup_{\left\Vert
u\right\Vert =1}\Rea \left\langle h^{\prime }(0)u,u^{\ast
}\right\rangle \leq 0$, that is, the linear mapping $-h^{\prime }(0)$
is also semi-complete on each ball $\mathcal{ B}_{R},$ $R>0.$
Moreover, if $L$ satisfies the stronger condition
\begin{equation}
L+4\left\Vert h\left( 0\right) \right\Vert <0,  \label{F1}
\end{equation}
then the following assertions hold.

(i) Let $$r_{1}=\frac{1}{2\left\Vert h\left( 0\right) \right\Vert
}\left[ -\left( 2\left\Vert h\left( 0\right) \right\Vert +L\right)
-\sqrt{\left( L+4\left\Vert h\left( 0\right) \right\Vert \right)
L}\right] (<1)$$ be the smaller root of the quadratic equation
\begin{equation}
\left\Vert h\left( 0\right) \right\Vert \left( 1+r\right)
^{2}+Lr=0. \label{F2}
\end{equation}
Then for each $r\in [r_{1},1]$, the mapping $-h$ is semi-complete
on the ball $\mathcal{B}_{r}.$

(ii) The mapping $h$ has a unique null point $x_{0}$ in
$\mathcal{B}$ with $\left\Vert x_{0}\right\Vert \leq r_{1}.$

(iii) If the family $\left\{ \Phi _{\lambda }\right\} _{\lambda
\geq 0}$ is defined by using the resolvent $\Re (\lambda ,h),$
\begin{equation*}
\Phi _{\lambda }:= \Re (\lambda ,h)\circ \left( \lambda I\right) :=\left(
\lambda I-h\right) ^{-1}\circ \left( \lambda I\right) ,
\end{equation*}%
then for each $\lambda >0$, the iterates $\left\{ \Phi _{\lambda
}^{n}\right\} _{n=1}^{\infty }$ converge to the constant mapping
taking the value $x_{0}$,
uniformly on each ball strictly inside $\mathcal{B}$.
\end{corollary}

\begin{proof}
First we note that the inequality \ $\mu (r)\leq 0$ is equivalent
to the inequality $\omega \left( r\right) =r\mu (r)=\frac{1-r}{1+r
}\left( \left\Vert h\left( 0\right) \right\Vert \left( 1+r\right)
^{2}+Lr\right) \leq 0.$ Therefore, $\mu (r)$ is negative on the
interval $ [r_{1},1)$ with $r_{1}<1$ if and only if condition
(\ref{F1}) holds. In this case $r_{1}$ is the smaller root of
equation (\ref{F2}). This prove assertion (i). To prove assertions
(ii) and (iii) we return to equation (\ref{4.12}) and recall that
it has a unique solution if and only if condition (\ref{4.14})
holds.\ \ Since for each $r\in $\ $(r_{1},1),$\
$\omega \left( r\right) $ is negative, one can set $\lambda =0$
and $z=0$ to obtain the existence and uniqueness of the solution
of the equation $h\left( x\right) =0.$ On the other hand, setting
in (\ref{4.12}) $z=\lambda y$, $\left\Vert y\right\Vert =r\left(
1-\mu (r)\right) :=R,$ we see that for each $r\in $\ $
(r_{1},1),$\ inequality (\ref{4.14})\ holds for each $\lambda >0$.
This means that for such $\lambda $ the mapping $\Phi _{\lambda
}=\left( \lambda I-h\right) ^{-1}\circ \left( \lambda I\right) $\
is well defined on the ball of radius $R$ and maps it into the
smaller ball of radius $r\in (r_{1},1).$ It then follows from the
Earle-Hamilton Theorem \cite{E-H} that \ $\Phi _{\lambda }$ has a
unique fixed point $x\left( \lambda \right) \in $\ $
\mathcal{B}_{r}$ and that its iterates $\left\{ \Phi _{\lambda
}^{n}\right\} _{n=1}^{\infty }$\ \ converge to \ $x\left( \lambda
\right)$, uniformly on each ball strictly inside
$\mathcal{B}_{R}.$\ Since \ $\mu (r)\rightarrow 0$ as\
$r\rightarrow 1^{-},$\ \ we see that $R\rightarrow 1.$ Hence
this convergence is uniform on each ball strictly inside
$\mathcal{B}$. Finally, we note that since $h\left( x_{0}\right)
=0,$ \ $\Phi _{\lambda }\left( x_{0}\right) =x_{0}.$
Hence \ $x\left( \lambda \right) =x_{0}$ does not depend on
$\lambda >0$\ \ because of the uniqueness property. The proof is
complete.
\end{proof}

\begin{corollary}
\label{fixp}Let $F:\mathcal{B\rightarrow B}$ be a holomorphic
self-mapping of $\mathcal{B}$ and assume that
\begin{equation}
L_{F}:=\sup_{\left\Vert u\right\Vert =1}\Rea \left\langle
F^{\prime }(0)u,u^{\ast }\right\rangle <1-4\left\Vert F\left(
0\right) \right\Vert . \label{Fix}
\end{equation}%
Then $F$ has a unique fixed point $x_{0}$ in $\mathcal{B}$ with
$\left\Vert x_{0}\ \right\Vert \leq r_{1,}$ where $r_{1}$ is the
fixed point in the interval $\left( 0,1\right) $ of the scalar
mapping\ $\phi \left( r\right) =\left\Vert F\left( 0\right)
\right\Vert \left( 1+r\right) ^{2}+rL_{F}.$\ \
\end{corollary}

\begin{remark}
\label{shwarz}Inequality (\ref{Fix})\ \ reminds us of the well-known
one-dimensional Schwarz inequality%
\begin{equation*}
\left\vert F^{\prime }\left( 0\right) \right\vert \leq
1-\left\vert F\left( 0\right) \right\vert ^{2}.
\end{equation*}%
However, even in the one-dimensional case the last inequality does
not imply the existence of an interior fixed point of $F.$\
\end{remark}

We conclude this section with a result on the holomorphic
extension of the associated resolvent mapping to  key domains of
the complex plane.

\begin{theorem}
\label{hol} Let $h$ be a semi-complete vector field on the open unit ball $%
\mathcal{B}$ in $X$ with $h\left( 0\right) =0$ and $h^{\prime }\left(
0\right) =-I.$ Then for each $r\in \left( 0,1\right) $ and $\lambda \in
\Omega $, where $\Omega =\Omega _{1}\cup \Omega _{2}$ is a key domain
defined by the disc $\Omega _{1}=\left\{ \lambda \in
\mathbb{C}
:\left\vert \lambda \right\vert <\frac{1}{2}\frac{1-r}{1+r}\right\} $ and
the sector  $\Omega _{2}=\left\{ \lambda \in
\mathbb{C}
,\text{ }\lambda \neq 0:\left\vert \arg \lambda \right\vert <\arcsin \frac{%
1-r^{2}}{1+r^{2}}\right\} ,$ the associated resolvent mapping
\begin{equation*}
\Phi _{\lambda }=(\lambda I-h)^{-1}\circ (\lambda I)
\end{equation*}%
is a self-mapping of the ball $\mathcal{B}_{r}.$ Moreover, this mapping is
holomorphic on $\Omega \times \mathcal{B}_{r}.$
\end{theorem}

\begin{proof}
First we show that $\Phi _{\lambda }$ is well defined on $\Omega_{1}$
and maps  $\mathcal{B}_{r}$ into itself, or which is one and the
same, that the equation
\begin{equation}
(\lambda I-h)(x)=\lambda y  \label{res}
\end{equation}
has a unique solution $x\in \mathcal{B}_{r}$ for each $\lambda \in
\Omega _{1}$ and $y\in \mathcal{B}_{r}$. To this end, we consider
the mapping
\begin{equation}
G(x)=\lambda y-\lambda x+h(x)  \label{mapG}
\end{equation}%
and show that
\begin{equation*}
\sup_{\left\Vert x\right\Vert =r}\Rea \left\langle G(x),x^{\ast
}\right\rangle <0.
\end{equation*}

Indeed, it follows from Proposition \ref{propT} that\ under our assumptions,
\begin{equation*}
\begin{split}
\sup_{\left\Vert  x\right\Vert =r}\Rea \left\langle G(x),x^{\ast
}\right\rangle &\leq\sup_{\left\Vert x\right\Vert =r}\Rea
\left\langle \lambda \left( y-x\right) ,x^{\ast }\right\rangle
+\frac{ r^{2}\left( 1-r\right) }{1+r}\sup_{\left\Vert u\right\Vert
=1}\Rea  \left\langle h^{\prime }\left( 0\right) u,u^{\ast
}\right\rangle  \\ &\leq \!\sup_{\left\Vert u\right\Vert
=1}\!\left\vert \lambda \right\vert \left\Vert y-x\right\Vert
r\!-\!\frac{ r^{2}\left( 1-r\right)
}{1+r}\!<\!\frac{1}{2}\frac{1-r}{1+r}2r^{2}\!-\!\frac{ r^{2}\left(
1-r\right) }{1+r}\!=\!0.
\end{split}
\end{equation*}
Thus $G$ has a unique null point $x=$\ $\Phi _{\lambda }\left(
y\right) $ in $\mathcal{B}_{r}$, as required.

Now assume that for some $r\in \left( 0,1\right) $, the
complex number $\lambda =\left\vert \lambda \right\vert e^{i\theta
}\in \Omega _{2}$ is given. \ Then
equation (\ref{res}) can be rewritten as
\begin{equation*}
(\left\vert \lambda \right\vert I-e^{-i\theta }h)(x)=\left\vert
\lambda \right\vert y.
\end{equation*}
In its turn, the last equation has a unique solution $x\in
\mathcal{B}_{r}$ for each $\lambda \in \Omega _{2}$ and $y\in
\mathcal{B}_{r}$\ whenever \ the mapping $e^{-i\theta }h$ is
semi-complete on \ $\mathcal{B}_{r}$ or, which is one and the same,
\begin{equation*}
\sup_{\left\Vert x\right\Vert =r}\Rea \left\langle e^{-i\theta
}h(x),x^{\ast }\right\rangle \leq 0.
\end{equation*}%
It follows from Proposition \ref{propN} that
\begin{equation}
\sup_{\left\Vert x\right\Vert =r}\Rea \left\langle e^{-i\theta
}h(x),x^{\ast }\right\rangle \leq r^{2}\left( \frac{2r\left(
1-r\cos \theta \right) }{1-r^{2}}-\cos \theta \right)
=\frac{r^{2}}{1-r^{2}}\varphi \left( r\right) \leq 0,  \label{fi}
\end{equation}%
as long as $\left\vert \theta \right\vert =\left\vert \arg \lambda
\right\vert <\arcsin \frac{1-r^{2}}{1+r^{2}},$ and we are done. To
finish the proof we just note that since the point  $x=0$ is a
regular null point of the mapping $G$\ \ defined by (\ref{mapG})
$\left( G^{\prime }\left( 0\right) =-I\right) ,$ it follows from
a version of the global implicit function theorem in \cite{Kh-R-S}
(see also Lemma \ref{lem1}  below) that the solution $x=$ $\Phi
_{\lambda }\left( y\right) $ of the equation\ \ $G\left( x\right)
\left( =G\left( x,\lambda ,y\right) \right) =0$ holomorphically
depends on $\left( \lambda ,y\right) $\ $\in $\ $\Omega \times
\mathcal{B}_{r}.$ This completes our proof.\ \
\end{proof}

\begin{remark}
\bigskip \label{Spiral} Actually, as we will see below (see Section 4),
the solution $r=r\left( \theta \right) =\frac{1-\left\vert \sin \theta
\right\vert }{\cos \theta }$ of the equation $\varphi \left(
r\right) =0,$ where  $\varphi \left( r\right) =2r\left( 1-r\cos
\theta \right) -\cos \theta \left( 1-r^{2}\right) $ is defined in
(\ref{fi}), determines the radius of spirallikeness for starlike mappings
defined on the unit ball $%
\mathcal{B}$.
\end{remark}

\ \ \ \ \ \ \ \ \ \ \ \ \ \ \ \ \ \ \ \ \ \ \ \ \ \ \ \ \ \ \ \ \ \ \ \ \ \
\ \ \ \ \ \ \ \ \ \ \ \ \ \ \ \ \ \ \ \ \ \ \ \ \ \ \ \ \ \ \ \ \ \ \ \ \ \
\ \ \ \ \ \ \ \ \ \ \ \ \ \ \ \ \ \ \ \ \ \ \ \ \ \ \ \ \ \ \ \ \ \ \

\section{Bloch radii}

Let $F:\mathcal{B}\rightarrow X$ be such that $F(0)=0$ and
$F^{\prime }(0)$ is an invertible operator on $X$. In other words,
$F$ is locally biholomorphic around the origin.

One says that the positive numbers $r$ and $\rho $ are Bloch radii
for $F$ if $F( \mathcal{B}_{r})\supseteq \mathcal{B}_{\rho }$ and
$F^{-1}:\mathcal{B}_{\rho }\rightarrow\mathcal{B}_{r}$
is a well-defined holomorphic mapping on $ \mathcal{B}_{\rho }$.

A deficiency of this definition is that the pair $(r,\rho)$ is not
uniquely defined. If, for example, we find the maximal $\rho$ for
which $F^{-1}$ is holomorphic on $\mathcal{B}_{\rho }$, then for each $%
\widetilde{r}\in [r,1]$, the pair $(\widetilde{r},\rho )$ constitutes
Bloch radii. However, in this case it is often desirable to find
the minimal $r$ for which $F(\mathcal{B}_{r})\supseteq
\mathcal{B}_{\rho }$.

Sometimes it is preferable to find a number $0<r_{*}\leq 1$ and a
continuous
function $\rho (r)$ on $[0,r_{*}]$ (if it exists) such that all the pairs $%
(r,\rho (r))$ are Bloch radii. In this case, one can investigate
the distortion (dilation) coefficient
\begin{equation*}
\varepsilon (r) =\frac{r}{\rho (r)}>0
\end{equation*}
on the interval $[0,r_{*}]$ and look for its bounds.

For example, if $F(0)=0$ and $F^{\prime }(0)=I$, then in the
one-dimensional case it follows from Koebe's $1/4$--theorem that
if $r_{\ast }$ is a radius of univalence of $F$ in $\mathcal{B}$,
then $\varepsilon (r)=4 $ for each $r\in (0,r_{\ast }]$.


In general, under the above normalization, the inverse function
theorem
shows that Bloch radii exist for $F$. In this case, one can write $%
F(x)=x-h(x)$, where $h^{\prime }(0)=0$. However, the latter
condition is not necessary: one can just require that $I-h^{\prime
}(0)$ be an invertible linear operator. In particular, in order
to get estimates in terms of the numerical range we can assume
that
\begin{equation*}
L=\sup\limits_{\Vert x\Vert =1}\Rea \langle h^{\prime }\left(
0\right)x ,x^{\ast }\rangle <1.
\end{equation*}

Consider the equation
\begin{equation}
x-h(x)=z,\text{ \ }z\in X,\text{ \ }\left\Vert x\right\Vert
<1\text{ .} \label{1}
\end{equation}

Our goal is to find numbers $0<r<1$ and $\rho \; (=\rho (r))$ such
that for all $z\in \mathcal{B}_{\rho }$, equation ($\ref{1}$) has
a unique solution $ x=x(z)\in \mathcal{B}_{r}$, which is
holomorphic in $z\in \mathcal{B}_{\rho }$.

If $N=\limsup\limits _{s\rightarrow 1^{-1}}\sup\limits_{\Vert x\Vert
=1}\Rea \langle h(sx),x^{\ast }\rangle <\infty $, then one can use
Proposition \ref{teor4.1} with $R=1$ and $\lambda =1$ to obtain
estimates for the Bloch radii. However, one can devise an
algorithm for finding lower bounds of Bloch radii under weaker
restrictions.

Let us assume that $h(0)=0$ and that for some $\theta \in \mathbb{R}$,
the mapping $h$ satisfies the condition%
\begin{equation}
\sup_{x\in B}\Rea \left\langle e^{i\theta }h(x),x^{\ast
}\right\rangle =N(\theta )<\infty \text{ .}  \label{2}
\end{equation}
Suppose that the following numbers are given:
\begin{equation}  \label{3}
L(\theta ):= \sup_{\left\Vert x\right\Vert =1}\Rea \left\langle
e^{i\theta }h^{\prime }(0)x,x^{\ast }\right\rangle \quad
\mbox{and} \quad l(\theta ):=\inf_{\left\Vert x\right\Vert =1}\Rea
\left\langle e^{i\theta }h^{\prime }(0)x,x^{\ast }\right\rangle .
\end{equation}
We let $L:=L(0)$.

Since the Fr\'{e}chet derivative of a holomorphic mapping is a
bounded linear operator, $L(\theta )$ and $l(\theta )$ are finite for all $%
\theta \in\mathbb{R}$.

We use the following version of the implicit function theorem

\begin{lemma}[\protect\cite{RS-SD-96} and \protect\cite{Kh-R-S}]
\label{lem1} Let $G \; (=G(x,z))$ be a holomorphic mapping in the
domain $ \mathcal{D=B}_{r}\times \mathcal{B}_{\rho }$  with values
in $X$ and assume that for each $z\in \mathcal{B} _{\rho }$ and
$x=su$, $\left\Vert u\right\Vert =r$, $0<s<1$,
\begin{equation*}
\limsup_{s\rightarrow 1^{-}}\Rea \left\langle G(su,z),u^{\ast
}\right\rangle <0\text{.}
\end{equation*}
Then

(i) for each $z\in \mathcal{B}_{\rho }$, there is a unique solution $%
x \,(=x(z))\in \mathcal{B}_{r}$ of the equation%
\begin{equation*}
G(x,z)=0,
\end{equation*}%
which holomorphically depends on $z\in \mathcal{B}_{\rho }$;

(ii) for each $z\in\mathcal{B}_{\rho}$, the linear operator
$G_{x}^{\prime }(x(z),z)$ is invertible in $X$.
\end{lemma}

We now consider the mapping $G:B\times X\rightarrow X$ defined by
\begin{equation}
G(x,z):=z-x+h(x)  \label{4}
\end{equation}%
and note that equation (\ref{1}) is equivalent to%
\begin{equation}
G(x,z)=0\text{ .}  \label{5}
\end{equation}

In view of Lemma \ref{lem1}, our aim becomes to find $r\in (0,1)$
and $\rho =\rho (r)>0$ such that the following inequality holds
whenever $\left\Vert x\right\Vert =r$ and $\left\Vert z\right\Vert
<\rho \; (=\rho (r))$:
\begin{equation}
\sup_{\left\Vert x\right\Vert =r}\Rea \left\langle G(x,z),x^{\ast
}\right\rangle <0.  \label{6}
\end{equation}
Equation (\ref{6}) is a sufficient condition for (\ref{5}) to have
a unique solution $x=x(z)$ in the ball $\mathcal{B}_{r}$, $r\in
(0,1)$.

Since%
\begin{equation}
\sup_{\left\Vert x\right\Vert =r}\Rea \left\langle G(x,z),x^{\ast
}\right\rangle \leq \left\Vert z\right\Vert \cdot
r-r^{2}+\sup_{\left\Vert x\right\Vert =r}\Rea \left\langle
h(x),x^{\ast }\right\rangle  \label{7},
\end{equation}%
we have, as a matter of fact, to use an appropriate growth estimate for the
last term in (\ref{7}).

Now it follows from Proposition \ref{propN} (with $R=1$) that

\begin{equation}  \label{8}
\sup_{\left\Vert x\right\Vert =r}\Rea \left\langle h(x),x^{\ast
}\right\rangle \leq r^{2}\left[ L+\mathcal{L}(\theta
,r)(N(\theta )-l(\theta ))\right].
\end{equation}
Returning to (\ref{7}), we finally get for $\left\Vert z\right\Vert
\leq \rho $ that
\begin{equation}
\sup_{\left\Vert x\right\Vert =r}\Rea \left\langle G(x,z),x^{\ast
}\right\rangle \leq r\left[ \rho -r\left( 1-L-\delta (\theta )\mathcal{L}%
(\theta ,r\right) )\right] \text{,}  \label{11}
\end{equation}%
where
\begin{equation}
\delta (\theta ):=N(\theta )-l(\theta )\geq 0  \label{12}
\end{equation}%
by Proposition \ref{prop1}. Moreover, $\delta (\theta )=0$ if and only if $%
h(x)=h^{\prime }(0)x$, $x\in \mathcal{B}$, is a restriction of the
bounded linear operator $h^{\prime }(0)$ in $X$ (see Corollary
\ref{cor(rigidity)}). So, in this case we have by (\ref{11}),
\begin{equation*}
\sup_{\left\Vert x\right\Vert =r}\Rea \left\langle G(x,z),x^{\ast
}\right\rangle \leq 0
\end{equation*}%
if
\begin{equation}
\rho \leq r(1-L),\text{ \ }r\in (0,1].  \label{13}
\end{equation}%
Thus for each $r\in(0,1]$ and $\rho (r)=r(1-L)$, the pair $(r,\rho
(r))$ constitutes Bloch radii, and $\varepsilon (r)=\frac{r}{\rho
(r)}=\frac{1}{1-L}$
is a constant function. Obviously, the function $\rho (r)$ is positive on $%
(0,1]$ and attains its maximum $\rho _{0}=1-L$ at the point
$r_{0}=1$ whenever $L<1$.

Now we assume that $\delta (\theta )> 0$, that is, $N(\theta
)>l(\theta )$. In this case
\begin{equation*}
\sup_{\left\Vert x\right\Vert =r}\Rea \left\langle G(x,z),x^{\ast
}\right\rangle \leq 0
\end{equation*}
if
\begin{equation}
\rho <\rho (r):=r(1-L-\delta (\theta )\mathcal{L}(\theta ,r)) ,
\quad \left\Vert z\right\Vert <\rho , \quad \left\Vert
x\right\Vert =r.  \label{16}
\end{equation}

It would, of course, be pertinent to look for conditions which ensure that $%
\rho (r)>0$ for some $r\in (0,1]$ and to find the maximum of this
function on this interval.

Writing down explicitly (\ref{16}), we get
\begin{equation}
\rho (r)=\frac{r}{1-r^{2}}\left( (1-L)(1-r^{2})-\delta (\theta
)2r(1-r\cos \theta )\right) .  \label{a}
\end{equation}%
As above we assume in the sequel that $L<1$. Since $\rho (0)=0$,
we have
$\rho ^{\prime }(0)=\lim_{r \rightarrow 0^{+}}\displaystyle\frac{\rho (r)%
}{r}=1-L>0.$ In addition, $\lim_{r\rightarrow 1^{-}}\rho
(r)=-\infty $ whenever $\theta \neq 0$, and for all $r\in (0,1)$,
\begin{equation*}
\rho ^{\prime \prime }(r)=\frac{4\delta (r^{3}\cos \theta
-3r^{2}+3r\cos \theta -1)}{(1-r^{2})^{3}}<\frac{-4\delta
}{(1+r)^{3}}<0.
\end{equation*}

Thus, again, the condition $L<1$ ensures that $\rho $ has a positive maximum $
\rho (r_{0})$ at some point $r_{0}\in (0,1)$.

If $\theta \neq 0$, it is clear that $r_{0}<r_{\ast }<1$, where
$r_{\ast }$ is the minimal (positive) root of the equation $\rho
(r)=0$ or, which is one and the same, of the equation
\begin{equation}
\varphi (r):=r^{2}\left( 2\delta (\theta )\cos \theta
-(1-L)\right) -2r\delta (\theta )+1-L=0.  \label{17}
\end{equation}

Since $\varphi (0)=1-L>0$ and $\varphi (1)=2\delta (\theta )(\cos
\theta -1)<0$ whenever $\theta \neq 0$, we see that if
\begin{equation}
2\delta (\theta )\cos \theta \neq 1-L,  \label{18}
\end{equation}%
then the unique root of equation (\ref{17}) in the interval (0,1) is
\begin{eqnarray}  \label{r-ast}
r_{\ast } &=&\frac{\delta (\theta )-\sqrt{\delta ^{2}(\theta
)+\left[ (1-L)-2\delta (\theta )\cos \theta \right]
(1-L)}}{2\delta (\theta )\cos \theta -1+L} \\ &=&\frac{\delta
(\theta )-\sqrt{\left[ \delta (\theta )-(1-L)\right] ^{2}+2\delta
(\theta )(1-L)(1-\cos \theta )}}{2\delta (\theta )\cos \theta
-(1-L)}\text{, \ \ \ }\theta \neq 0,  \notag
\end{eqnarray}%
because the numerator and denominator of the last expression have
the same sign.

Finally, if
\begin{equation}
2\delta (\theta )\cos \theta =1-L  \label{19},
\end{equation}%
we see that%
\begin{equation}
\rho (r)=\frac{r}{1-r^{2}}(1-L-2r\delta (\theta )) \label{funcro}
\end{equation}%
and
\begin{equation}
r_{\ast }=\frac{1-L}{2\delta (\theta )}=\cos \theta <1 \label{20}
\end{equation}%
whenever $\theta \neq 0$.

\begin{proposition}
\label{bl-ra} Let $h:\mathcal{B}\rightarrow X$ be holomorphic with
$h(0)=0$, let the functions $N(\theta )$, $l(\theta )$ and $L(\theta
)$ be defined by (\ref{2}) and (\ref{3}), and let $\rho (r)$ be
defined by (\ref{a}). Then for all $r\in (0,r_{\ast })$, where
$r_{\ast }$ is defined by (\ref{r-ast}) (or (\ref{20}) in the case of
(\ref{19})), the numbers $r$ and $\rho (r)$ are Bloch radii for the
mapping $F=I-h.$ Moreover, the equation $\rho ^{\prime }(r)=0$ has
a unique solution $r_{0}\in (0,r_{\ast })$, and so the function
$\rho (r)$ attains its maximum $\rho _{0}$ at this interior point
$r_{0}\in (0,r_{\ast })\subset \left( 0,1\right).$
\end{proposition}

To find some explicit estimates for $r_{0}$ and $\rho _{0}$, we
exploit again Proposition \ref{propN}, but using another approach in order
to simplify our calculations. Namely, applying Proposition \ref{propN}
with $R=1$, we see
that for any fixed $s\in (0,r_{\ast })$,
\begin{equation*}
N_{s}(h) =\sup_{\left\Vert x\right\Vert =s}\Rea \left\langle
h(x),x^{\ast }\right\rangle   \label{24} \leq s^{2}\left[ L+\mathcal{L(\theta
},s)\delta (\theta )\right],
\end{equation*}%
where $\delta (\theta )=N(\theta )-l(\theta )$, where $N(\theta )$,
$l(\theta ) $ and $L$ are given by (\ref{2}) and (\ref{3}).

On the other hand, if we set $R=s$  in Remark \ref{lem},  we get that for any $x$ such that
$\left\Vert x\right\Vert =r<s$,
\begin{eqnarray}  \label{a27}
\Rea \left\langle h(x),x^{\ast }\right\rangle &\leq
&r^{2}\sup_{\left\Vert u\right\Vert =1}\Rea \left\langle h^{\prime
}(0)u,u^{\ast }\right\rangle \cdot \left(1-\mathcal{L}_s(0,r)\right) +\frac{%
r^{2}}{s^{2}}\mathcal{L}_{s}(0,r)N_{s}(h)  \notag \\
&\leq &r^{2}\left[ L\left( 1-\frac{2r}{s+r}\right) +\frac{1}{s^{2}}\frac{2r}{%
s+r}\cdot N_{s}(h)\right]  \notag \\
&=&r^{2}\left[ L\frac{s-r}{s+r}+\frac{2r}{s+r}\cdot \frac{N_{s}(h)}{s^{2}}%
\right]  \notag \\ &\leq &r^{2}\left[
L\frac{s-r}{s+r}+\frac{2r}{s+r}\left[ L+\mathcal{L(\theta
},s)\delta (\theta )\right] \right].
\end{eqnarray}

To simplify further our calculations we denote
\begin{equation}
\begin{split}
&K \,(=K(\theta ,s)):=L+\mathcal{L(\theta },s)\delta (\theta )=L+Q,
\label{26} \\ &Q \, (=Q(\theta ,s)):=\mathcal{L(\theta },s)\delta
(\theta )>0,  \notag
\end{split}
\end{equation}
and set $R=s<1$.

Then (\ref{a27}) becomes
\begin{equation}
\sup_{\left\Vert x\right\Vert =r}\Rea \left\langle h(x),x^{\ast
}\right\rangle \leq r^{2}\left[ \frac{s-r}{s+r}L+\frac{2r}{s+r}K\right] =%
\frac{r^{2}}{s+r}\left[ (s-r)L+2rK\right].  \label{27}
\end{equation}

Now we again consider inequality (\ref{7}) for $r\in (0,s]$, taking
into account (\ref{27}). We then obtain
\begin{equation*}
\sup_{\left\Vert x\right\Vert =r}\Rea \left\langle G(x,z),x^{\ast
}\right\rangle \leq \left\Vert z\right\Vert \cdot r-r^{2}+\frac{r^{2}}{s+r}%
\left[ (s-r)L+2rK\right] <0
\end{equation*}%
whenever
\begin{equation}
\left\Vert z\right\Vert <\rho _{s}(r)=\frac{r^{2}(1+L-2K)+rs(1-L)}{s+r}=%
\frac{Ar^{2}+Br}{s+r},  \label{28}
\end{equation}%
where
\begin{equation*}
A(=A(s))=1+L-2K=1+L-2(L+Q)=1-L-2Q
\end{equation*}%
and%
\begin{equation}
B \, (=B(s))=s(1-L).  \label{29}
\end{equation}%
It is important to observe that for each fixed $s\in \left(
0,r_{\ast }\right)$, the following relations hold:

(i) $\rho (s)\geq \rho _{s}(r),$

\noindent while

(ii) $\rho (s)=\rho _{s}(s).$

However, the investigation of the function $\rho _{s}(r)$, $r\in
(0,s],$ in order to find its maximum value can be done explicitly via
quadratures.

Consider the function
\begin{equation*}
A(s)=1-L-2Q(s)=\frac{(4\delta (\theta )\cos \theta
-(1-L))s^{2}-4\delta (\theta )s+(1-L)}{1-s^{2}}.
\end{equation*}%
Note that $A(0)=1-L>0$, while
\begin{equation}
\begin{split}
A(r_{\ast })& =\frac{(1-L)(1-r_{\ast }^{2})-4\delta (\theta
)r_{\ast }(1-r_{\ast }\cos \theta )}{1-r_{\ast }^{2}} \\ &
=\frac{(1-L)(1-r_{\ast }^{2})-2\delta (\theta )r_{\ast }(1-r_{\ast
}\cos \theta )}{1-r_{\ast }^{2}}-\frac{2\delta (\theta )r_{\ast
}(1-r_{\ast }\cos \theta )}{1-r_{\ast }^{2}} \\ & =\frac{\rho
(r_{\ast })}{r_{\ast }}-\frac{2\delta (\theta )r_{\ast }(1-r_{\ast
}\cos \theta )}{1-r_{\ast }^{2}}=-\frac{2\delta (\theta )r_{\ast
}(1-r_{\ast }\cos \theta )}{1-r_{\ast }^{2}}<0,
\end{split}%
\end{equation}%
where $r_{\ast }$ is the unique positive root of the equation
$\rho (r)=0$ in $(0,1)$ defined by (\ref{r-ast}). So the minimal
positive root
\begin{equation*}
s_{\ast }:=\frac{2\delta (\theta )-\sqrt{4\delta (\theta
)^{2}-\left( 4\delta (\theta )\cos \theta -(1-L)\right)
(1-L)}}{4\delta (\theta )\cos \theta -(1-L)}
\end{equation*}%
of the equation $A(s)=0$ belongs to $(0,r_{\ast })$. Moreover, for
$s\in (0,s_{\ast })$, $A(s)>0$ or, which is one and the same,
$Q<\frac{1-L}{2}$, and for $%
s\in (s_{\ast },r_{\ast })$, $A(s)<0$ $\left(
Q>\frac{1-L}{2}\right) $.

As we have mentioned above, for all $s\in (0,1)$, we have $\rho
(s)=\rho _{s}(s)$. In particular, $\rho (s_{\ast })=\rho _{s_{\ast
}}(s_{\ast })=\displaystyle\frac{ B(s_{\ast
})}{2}=\displaystyle\frac{s_{\ast }(1-L)}{2}>0$.

Since $A(s_{*})=0$, the inequality
\begin{equation}  \label{supre}
\sup\limits_{\|x\|=r}\Rea  \langle G(x,z),x^{*}\rangle <0
\end{equation}
holds whenever
\begin{equation*}
\|z\| <\rho _{s_{*}}(r)=\frac{s_{*}(1-L)r}{s_{*}+r}, \quad
\|x\|=r\in (0,s_{*}].
\end{equation*}
Note that in this case, $\varepsilon (r)=\displaystyle\frac{r}{\rho (r)}=\displaystyle\frac{s_{*}+r}{%
B(s_{*})}$ is an affine function.

Since the derivative
\begin{equation*}
\rho _{s_{\ast }}^{\prime }(r)=\frac{2s_{\ast }^{2}Q}{(s_{\ast }+r)^{2}}%
>0\quad (Q>0),
\end{equation*}%
the increasing function $\rho _{s_{\ast }}(r)$, $r\in
\lbrack 0,s_{\ast }]$, attains its maximum on $[0,s_{\ast }]$ at
the point $s_{\ast }$, that is,
\begin{equation*}
\max\limits_{r\in \lbrack 0,s_{\ast }]}\rho _{s_{\ast }}(r)=\rho
_{s_{\ast }}(s_{\ast })=\rho (s_{\ast })=\frac{(1-L)s_{\ast }}{2},
\end{equation*}%
and the pair $\left( s_{\ast },\frac{(1-L)s_{\ast }}{2}\right) $
constitutes Bloch radii for $F$, with $\varepsilon (s_{\ast
})=\frac{2}{1-L}$.

Now we fix $s\in (0,s_{\ast })$. In this case, inequality
(\ref{supre}) holds whenever
\begin{equation*}
\Vert z\Vert <\rho _{s}(r)=\frac{A(s)r^{2}+B(s)r}{s+r},\quad r\in
(0,s).
\end{equation*}%
The derivative
\begin{equation*}
\frac{d\rho
_{s}(r)}{dr}=\frac{A(s)(s+r)^{2}+2s^{2}Q}{(s+r)^{2}}>0,\quad r\in
(0,s),
\end{equation*}%
because $A(s)>0$ for $s\in (0,s_{\ast })$. Hence, $\rho $ is increasing on $%
[0,s]$, attains its maximum on $[0,s]$ at the point $s$, that is,
\begin{equation*}
\max\limits_{r\in \lbrack 0,s]}\rho _{s}(r)=\rho _{s}(s)=\rho (s)=(1-L)s-%
\frac{2\delta (\theta )s^{2}(1-s\cos \theta )}{1-s^{2}},
\end{equation*}%
and for each $s\in (0,s_{\ast })$, the pair
$\left(s,(1-L)s-\frac{2\delta (\theta )(1-s\cos \theta )}{1-s^{2}}\right)$
constitutes Bloch radii for $F$.

Next we fix $s\in (s_{\ast },r_{\ast })$ and let
\begin{equation*}
\Vert z\Vert <\rho _{s}(r)=\frac{A(s)r^{2}+B(s)r}{s+r},\quad r\in
(0,s).
\end{equation*}%
In this case, $A(s)<0$ and so the equation
\begin{equation*}
\frac{d\rho _{s}(r)}{dr}=\frac{A(s)(s+r)^{2}+2s^{2}Q}{(s+r)^{2}}=0
\end{equation*}%
or, which is one and the same,
\begin{equation*}
(s+r)^{2}=-\frac{2s^{2}Q}{A(s)},
\end{equation*}%
makes sense. It can be seen that its minimal positive solution
\begin{equation*}
r^{0}=\left( \sqrt{\frac{2Q}{2Q-(1-L)}}-1\right) s
\end{equation*}%
belongs to $(0,s)$ if and only if $Q>\frac{2}{3}(1-L)$. Since
$\rho _{s}(0)=0 $, $\left .\frac{d\rho
_{s}(r)}{dr}\right|_{r=0^{+}}=1-L>0$ and $\frac{d^{2}\rho
_{s}(r)}{dr^{2}}=- \frac{4s^{2}Q}{(s+r)^{3}}<0$ in $(0,s)$, the
function $\rho _{s}(r)$ attains its maximum on $[0,s]$ at the
point $r^{0}$, is positive and increasing on $[0,r^{0}]$, and the
pair $\left( r^{0},\rho _{s}(r^{0})\right) $ constitutes Bloch radii
for $F$.

If $Q\leq \frac{2}{3}(1-L)$, $\rho _{s}^{\prime }(r)$ does not
vanish in $ (0,s)$, and since $\rho _{s}^{\prime }(0)=1-L>0$,
$\rho _{s}$ is increasing on $[0,s]$ and attains its maximum on
$[0,s]$ at the point $s$, that is,
\begin{equation*}
\max\limits_{r\in \lbrack 0,s]}\rho _{s}(r)=\rho _{s}(s)=\rho
(s)=s(1-L-Q)\geq \frac{s}{3}(1-L)>0,
\end{equation*}%
and the pair $\left( s,s(1-L-Q(s))\right) $ constitutes Bloch radii
for $F$.

Now we summarize our conclusions in the following assertion.

\begin{proposition}
\label{bloch2}Let $F=I-h$, where $h\in
\mathrm{Hol}(\mathcal{B},X)$ with $ h(0)=0$, the functions $N(\theta
)$, $l(\theta )$ and $L(\theta )$ be defined by (\ref{2}) and
(\ref{3}), and let $L=L\left( 0\right) <1$. Then for each $s\in
\left( 0,r_{\ast }\right) ,$ where $r_{\ast }$ is defined by (\ref
{r-ast}) (or (\ref{20}) in the case of (\ref{19})), the function $\rho
_{s}(r)$ defined by (\ref{28}) is positive on the interval $\left(
0,r_{\ast }\right) $ and satisfies the conditions $\rho (s)\geq
\rho _{s}(r)$ and  $\rho (s)=\rho _{s}(s).$ Hence the pair $\left(
r,\rho _{s}(r)\right) $\ constitutes Bloch radii for $F$. Moreover, the
following assertions hold.

$\quad $ (a) if
\begin{equation*}
s_{\ast }:=\frac{2\delta (\theta )-\sqrt{4\delta (\theta
)^{2}-\left( 4\delta (\theta )\cos \theta -(1-L)\right)
(1-L)}}{4\delta (\theta )\cos \theta -(1-L)},
\end{equation*}%
then $\varepsilon (r)=\displaystyle\frac{r}{\rho _{s_{\ast
}}\left( r\right) }$ is an affine function, namely, $\varepsilon
(r)=\displaystyle\frac{ s_{\ast }+r}{s_{*} (1-L)}$;

$\quad$ (b) if $s\in (0,s_{*})$, then $\rho _{s}(r)$, $r\in (0,s]$,
is strictly increasing and hence,
\begin{equation*}
\max\limits_{r\in (0,s]}\rho _{s}(r)=\rho _{s}(s)=\rho (s) =(1-L)s-\frac{%
2\delta (\theta ) s^{2}(1-s\cos \theta )}{1-s^{2}};
\end{equation*}

$\quad$ (c) if $s\in (s_{*},r_{*})$, then
\begin{equation*}
\max\limits_{r\in (s_{*},s]}\rho _{s}(r)=%
\begin{cases}
\rho _{s}(r^{0}), & Q>\displaystyle\frac{2}{3}(1-L) \\ \quad &  \\
\rho (s), & Q\leq \displaystyle\frac{2}{3}(1-L)
\end{cases}%
,
\end{equation*}
$\quad$ $\quad$ $\quad$ where $r^{0}=\left(\sqrt{\displaystyle\frac{2Q}{%
2Q-(1-L)}}-1\right)s$.
\end{proposition}

Note that the estimate $\displaystyle\frac{(1-L)s_{\ast}}{2}\leq
\rho _{0}$ (where $\rho _{0}$ is the Bloch radius given by
Proposition \ref{bl-ra}) is sharp as the following example shows.

\begin{example}
For $\theta =\displaystyle\frac{\pi}{3}$, $L=0$ and $\delta
(\theta )=1$, we
have $\rho (r)=\displaystyle\frac{r(1-2r)}{1-r^{2}}$ and so $r_{\ast}=%
\displaystyle\frac{1}{2}$.

Then
\begin{equation*}
\rho _{s}(r)=\frac{Ar^{2}+Br}{s+r},
\end{equation*}
where $A=1-2Q$ with $Q=\displaystyle\frac{s(2-s)}{1-s^{2}}$. Hence $A=%
\displaystyle\frac{s^{2}-4s+1}{1-s^{2}}$ and the unique zero of
$A$ in $(0,1) $ is
$s_{\ast}=2-\sqrt{3}<r_{\ast}=\displaystyle\frac{1}{2}$.

For $0<s<s_{\ast}$, $A(s)>0$ $\left(0<Q<\frac{1}{2}\right)$ and,
in this case,
\begin{equation*}
\rho_{s}^{\prime }(r)=\frac{A(s+r)^{2}+2s^{2}Q}{(s+r)^{2}}\geq 0.
\end{equation*}
Consequently, the function $\rho _{s}$ is increasing on $(0,s)$
and
\begin{equation*}
\max \limits_{r\in[0,s]}\rho _{s}(r)=\rho _{s}(s)=\rho (s) =\frac{s(1-2s)}{%
1-s^{2}}.
\end{equation*}

Furthermore,
\begin{equation*}
\rho ^{\prime }(r)=\frac{(r-2)^{2}-3}{(1-r^{2})^{2}}
\end{equation*}%
and so the unique maximum of $\rho $ in $(0,1)$ is achieved at $r_{0}=2-%
\sqrt{3}$ and equals
\begin{equation*}
\rho (2-\sqrt{3})=\frac{2-\sqrt{3}}{2}.
\end{equation*}%
On the other hand, it is easy to calculate that $s_{\ast }=2-\sqrt{3}$ and $%
\rho _{s_{\ast }}(s_{\ast })=\frac{2-\sqrt{3}}{2}=\rho
(2-\sqrt{3})$.
\end{example}

\begin{remark}
\label{rem3}

If $\theta =0$ the factor $e^{i\theta }- \mathcal{L}(\theta ,r)=1-\frac{2r}{1+r%
}=\frac{1-r}{1+r}$ in (\ref{8}) is real and this estimate
can, in fact, be replaced with a sharper one than (\ref{8}), namely,
\begin{equation}
\sup_{\left\Vert x\right\Vert =r}\Rea \left\langle h(x),x^{\ast
}\right\rangle \leq r^{2}\left[
\frac{1-r}{1+r}L+N\frac{2r}{1+r}\right] \text{,}  \label{10}
\end{equation}
so that we do not need in this case the value%
\begin{equation*}
l(0)=\inf_{\left\Vert x\right\Vert =1}\Rea \left\langle h^{\prime
}(0)x,x^{\ast }\right\rangle \leq L\text{.}
\end{equation*}
In its turn formula (\ref{a}) becomes
\begin{eqnarray*}
\rho (r) &=&\frac{r}{1-r^{2}}\left( (1-L)(1-r^{2})-\delta (\theta
) 2r(1-r)\right) = \\ &=&\frac{r}{1+r}\left( (1-L)(1+r)-2\delta
(\theta ) r\right) \text{.}
\end{eqnarray*}
In this situation $\rho (0)=0$, $\rho (1)=1-L-\delta (\theta ) =1-N$ and $%
\rho ^{\prime }(0)=1-L>0$.

So, $\rho (r)$ vanishes in $(0,1)$ if and only if
\begin{equation*}
N>1.
\end{equation*}

Otherwise, if
\begin{equation*}
N\leq 1,
\end{equation*}
then $\rho (r)\geq 0$ for all $r\in (0,1)$.

In particular, if $N\leq \displaystyle\frac{1}{2}(1+L)<1$, then
also $\rho
^{\prime }(r)>0 $ for all $r\in (0,1)$ and%
\begin{equation}
\max_{r\in \lbrack 0,1]}\rho (r)=\rho (1)=1-N>0\text{.}
\label{43}
\end{equation}
Finally, if $N=1$, then
\begin{equation*}
\rho (r)=\frac{r}{1+r}\left( (1-L)(1+r)-2(1-L)r\right) =\frac{r(1-r)}{1+r}%
(1-L)
\end{equation*}%
and%
\begin{equation}
\max_{r\in \lbrack 0,1]}\rho (r)=\rho (\sqrt{2}-1)=(\sqrt{2}-1)^{2}(1-L)%
\text{.}  \label{44}
\end{equation}
\end{remark}

In general, setting in this case $s=1$, we arrive at the following
assertion.

\begin{proposition}
\label{propA1} Let $h$ be a holomorphic mapping on the unit ball
$\mathcal{B} $ of $X$ with $h(0)=0$ and
\begin{equation*}
L=\sup_{\left\Vert x\right\Vert =1}\Rea \left\langle h^{\prime
}(0)x,x^{\ast }\right\rangle <1\text{.}
\end{equation*}
If
\begin{equation*}
N=\sup_{\left\Vert x\right\Vert <R}\Rea \left\langle h(x),x^{\ast
}\right\rangle \text{,}
\end{equation*}%
then

(i) $\ \ N\geq L$;

(ii) \ Bloch radii $r_{0}$ and $\rho _{0}$ for $h$ can be given by
\begin{equation*}
r_{0}=\left\{
\begin{array}{c}
\sqrt{\frac{2(L-N)}{1+L-2N}}-1,\text{ \ }\text{if \ }N\geq
\frac{2+L}{3}, \\
\\
1,\text{ \ otherwise}%
\end{array}%
\right.
\end{equation*}%
and%
\begin{equation*}
\rho _{0}=\left\{
\begin{array}{c}
\rho (r_{0})\text{, \ if \ }N\geq \frac{2+L}{3} \\
1-N\text{, \ otherwise.}%
\end{array}%
\right.
\end{equation*}
\end{proposition}

\begin{remark}
\label{remA2} If, in particular, $h^{\prime }\left( 0\right) =0$,
then Proposition \ref{propA1} coincides with Theorem 7 in
\cite{HRS}. Note, however, that the estimates in that theorem are
still true when $h^{\prime }\left( 0\right) \neq 0,$ but
$L=\sup_{\left\Vert x\right\Vert =1}\Rea  \left\langle h^{\prime
}(0)x,x^{\ast }\right\rangle =0.$
\end{remark}

\section{\protect\bigskip Radii of starlikeness and spirallikeness}

\begin{definition}
\label{spiral}Let $\mu $ be a complex number with $\Rea \mu
>0$. For a domain $\mathcal{D}$ in $X$, where $0\in\mathcal{D}$, a locally biholomorphic mapping $f:\mathcal{D}\to
X$, where $f\left( 0\right) =0, \mathbb{\ }$ is said to be $\mu
$-spirallike on $\mathcal{D}$ if for each $ y\in f\left(
\mathcal{D}\right) $ and $t\geq 0$, the curve $\exp \left\{ -\mu
t\right\} y$ is contained in $f\left( \mathcal{D}\right) .$ If, in
particular, $ \mu $ is a positive real number, then $f$ is said to
be starlike on $\mathcal{D }$.
\end{definition}

Below we discuss the following two problems.

1. \emph{Let }$\mu \in \mathbb{C} $\emph{\ with }$\Rea \mu
>0$\emph{\ and }$\arg \mu \in \left( 0,\frac{ \pi }{2}\right)
$\emph{\ be given. Let }$\mathcal{B\ }$\emph{be the open unit ball
in }$X$\emph{\ and let }$f$\emph{\ }$:\mathcal{B}\to X,$\emph{ \
where }$f\left( 0\right) =0,$\emph{\ be a }$\mu $\emph{-spirallike
mapping on }$ \mathcal{B}$.\emph{\ Find }$r\in (0,1)$\emph{\
(depending on }$\mu $\textit{)}\emph{\  such that }$f$\emph{\ is
starlike on the ball }$\mathcal{B}_{r}.$

2.\textit{\ Conversely, let }\emph{\ }$f$\emph{\ }$:\mathcal{B}\to
X,$ \textit{\ where }$f\left( 0\right) =0,$\emph{\ be a
}\textit{starlike}\emph{\ mapping on }$\mathcal{B}$.
\textit{Given}\emph{\ }$\mu \in \mathbb{C} $\emph{\ with }$\Rea
\mu >0$\emph{\ and }$\arg \mu \in \left( -\frac{ \pi
}{2},\frac{\pi }{2}\right) $, find $r\in \left( 0,1\right) $ such
that $f$ is $\mu $\textit{-spirallike on }$\mathcal{B}_{r}$.

To solve these two problems we first observe that a locally
biholomorphic mapping $f:\mathcal{D}\to X$, where $f\left(
0\right) =0,\mathbb{\ }$ is $\mu $-spirallike on $\mathcal{D}$ if
and only if it satisfies the following differential equation:
\begin{equation}
\mu f\left( x\right) =f^{\prime }\left( x\right) h\left( x\right)
, \label{spir}
\end{equation}%
where $h:\mathcal{D}\to X$ is a semi-complete vector field on $
\mathcal{D}$ (see, for example, \cite{ReichS}).

Since \emph{\ }$f\left( 0\right) =0$ and $f^{\prime }\left(0\right)$
is an invertible linear operator, we get that \emph{\ }
\begin{equation}
h\left( 0\right) =0\text{  and  }h^{\prime }\left( 0\right) =\mu I.
\label{initial}
\end{equation}

Also, it is known that if $\mathcal{D}$ is a convex domain in $X,$
then the set of holomorphic semi-complete vector fields is a real
cone. Therefore, in this case we can set without loss of
generality $\left\vert \mu \right\vert =1.$ Thus, setting $\theta
=\arg \mu $ $\left( \left\vert \mu \right\vert =1\right) $, we can
reformulate our problem as follows.

Let $f:\mathcal{B}\to X$ be a
locally biholomorphic mapping on $\mathcal{B}$ such that $f\left( 0\right) =0$  which satisfies
equation (\ref{spir}) with $\left\vert \mu \right\vert =1$,
$\theta =\arg \mu \in \left( 0,\frac{\pi }{2}\right)$, and let $h$
$:\mathcal{B}\to X$ satisfy
\begin{equation}
\Rea \, \langle h(x),\text{ }x^{\ast }\rangle \geq 0,\text{ } x\in
\mathcal{B}, \; x^{\ast}\in J(x)\text{.}  \label{genin1}
\end{equation}

\emph{Find }$r\in (0,1)$\emph{\ (depending on }$\theta
$\textit{)}\emph{\ and a mapping }$h_{1}:\mathcal{B}_{r}\to X$
\emph{with}
\begin{equation}
\Rea \, \langle h_{1}(x),\text{ }x^{\ast }\rangle \geq 0,\text{
}x\in \mathcal{B}_{r}\text{,} \; x^{\ast}\in J(x), \label{genin2}
\end{equation}%
\emph{such that} $f$\emph{\ also satisfies the equation }%
\begin{equation*}
f\left( x\right) =f^{\prime }\left( x\right) h_{1}\left( x\right)
\end{equation*}%
\emph{whenever} $x\in \mathcal{B}_{r}.$

It is clear that due to the uniqueness property of holomorphic
solutions of differential equations, the mapping $h_{1}$ must
equal $e^{-i\theta }h$, whence we have that for all $x\in
\mathcal{B}$ and $ x^{\ast}\in J(x)$,
\begin{equation*}
\Rea \, \langle e^{i\theta }h_{1}(x),\text{ }x^{\ast }\rangle
\geq 0.
\end{equation*}%
Thus to solve our problem we can use the estimates obtained in
Proposition \ref{prop1}.

Indeed, by Proposition \ref{prop1}, replacing $h$ by $h_{1}$ and
setting $R=1$ on the left-hand side of (\ref{4.60}), for all $x\in
\mathcal{B}$ such that $\left\Vert x\right\Vert =r<1$ it holds
\begin{equation}
\Rea \,\langle h_{1}(x),\text{ }x^{\ast }\rangle \geq
r^{2}\left( l_{1}(0)+\mathcal{L} _{1}(\theta )\left(
m_{1}(1,\theta )-L_{1}(\theta )\right) \right),  \label{starin}
\end{equation}
where
\begin{equation*}
l_{1}(0)=\inf_{\left\Vert u\right\Vert =1}\Rea \left\langle
h_{1}^{\prime }(0)u,u^{\ast }\right\rangle =1,
\end{equation*}
because by (\ref{initial}), $h_{1}^{\prime }(0)=e^{-i\theta
}h^{\prime }\left( 0\right) =I$. Moreover, by (\ref{genin1}),
\begin{equation*}
m_{1}(r,\theta )=\inf_{\left\Vert x\right\Vert =r}\Rea
\left\langle e^{i\theta }h_{1}(x),x^{\ast }\right\rangle
=\inf_{\left\Vert x\right\Vert =r}\Rea \left\langle h(x),x^{\ast
}\right\rangle \geq 0.
\end{equation*}%
\begin{equation*}
L_{1}(\theta )=\sup_{\left\Vert u\right\Vert =1}\Rea \left\langle
e^{i\theta }h_{1}^{\prime }(0)u,u^{\ast }\right\rangle =\cos
\theta
\end{equation*}%
and%
\begin{equation*}
\mathcal{L} \left( \theta \right) =\frac{2r\left( 1-r\cos \theta
\right) }{ 1-r^{2}}.
\end{equation*}%
So, inequality (\ref{starin}) holds as soon as
\begin{equation*}
\Rea \, \langle h_{1}(x),\text{ }x^{\ast }\rangle \geq
r^{2}\left( 1- \frac{2r\left( 1-r\cos \theta \right)
}{1-r^{2}}\cos \theta \right) :=r^{2}\varphi \left( r,\theta
\right)
\end{equation*}
for $x\in \mathcal{B}$, $\left\Vert x\right\Vert =r<1$, $
x^{\ast}\in J(x)$ and $0<\theta <\frac{ \pi }{2}$.

Now calculations show that $\varphi \left( r,\theta \right) \geq
0$ if and only if
\begin{equation*}
r\leq r^{\ast }\left( \theta \right) =\left[ \sqrt{2}\cos \left(
\theta - \frac{\pi }{4}\right) \right] ^{-1}<1,
\end{equation*}%
whenever $\left\vert \theta \right\vert <\frac{\pi }{2}.$

Thus we have proven the following result.

\begin{theorem}
\label{star-radius}Let $\mathcal{B}$ be the open unit ball in $X$
and let $ f $ $:\mathcal{B}\to X,$  $f\left( 0\right) =0$, be a
$\mu $-spirallike
mapping on\emph{\ }$\mathcal{B}$ with $\mu =e^{i\theta },$\ $-\frac{\pi }{2}%
<\theta <\frac{\pi }{2}$\emph{. }Then $f$ is starlike on the ball $\mathcal{B%
}_{r}$ for each $r\leq r^{\ast }\left( \theta \right) =\left[
\sqrt{2}\cos \left( |\theta |-\frac{\pi }{4}\right) \right]
^{-1}<1.$
\end{theorem}

\begin{remark}
\label{goluzin}The classical one-dimensional result of Grunsky
(see, for example, \cite{GG}) asserts that any univalent function
on the open unit disc is starlike on the disc centered at the
origin with radius $r_{\ast }=\tanh \frac{\pi }{4}\simeq 0.65.$ It
is clear that in the special case of spirallike functions we have
obtained a better estimate since $\min r^{\ast
}\left( \theta \right) =r^{\ast }\left( \frac{\pi }{4}\right) =\frac{1}{%
\sqrt{2}}\simeq 0.71>$ $r_{\ast }.$ In the one-dimensional case
this estimate was obtained by Robertson \cite{rob}.
\end{remark}

To solve the second problem described above we just use Theorem
\ref{hol} and Remark \ref{Spiral} to obtain the following result.

\begin{theorem}
\label{Spiral2}Let $\mathcal{B}$ be the open unit ball in $X$ and
let $f$ $: \mathcal{B}\to X,$ $f\left( 0\right) =0$, be a starlike
mapping on\emph{ \ }$\mathcal{B}$. Then for each $-\frac{\pi
}{2}<\theta <\frac{\pi }{2}$ and for each $0<$ $r\leq r\left(
\theta \right) =\frac{1-\left\vert \sin \theta \right\vert }{\cos
\theta }<1$, the mapping $f$ is $\mu $-spirallike on the ball
$\mathcal{B}_{r}$  with $\mu =e^{i\theta }$.
\end{theorem}

The following example suggested by one of the referees shows that
the estimates in Theorems \ref{star-radius} and \ref{Spiral2} are
sharp.

\begin{example}
Let $u\in X$ be a unit vector and let $M$ be a subspace of $X$
such that each $x\in X$ has a unique representation
$x=x_{1}u+\widehat{x}$, where $x_{1}\in \mathbb{C}$ and
$\widehat{x}\in M$. Consider the mappings
$f_{\theta}:\mathcal{B}\rightarrow X$,  defined by
$f_{\theta}(x)=\frac{1}{(1-x_{1})^{1+exp(2i\theta)}}x$, $\theta
\in [0,\frac{\pi}{2})$. It can be seen that $f_{0}$ is starlike on
$\mathcal{B}$ while $f_{\theta}$, $\theta \in (0,\frac{\pi}{2})$,
is $\mu$-spirallike on $\mathcal{B}$ with $\mu =e^{i\theta}$.
Calculations show that for these mappings the estimates given in
Theorems \ref{star-radius} and \ref{Spiral2} cannot be improved.
\end{example}

\text{}

\text{}

{\bf Acknowledgments.} We are very grateful to the anonymous referees
 for their many useful comments and helpful suggestions.

\text{}

\text{}

\bigskip

\bigskip

\bigskip

\text{}

\text{}

{\large Filippo Bracci, Dipartimento di Matematica, Universit\`{a}
di Roma ``Tor Vergata'', Via della Ricerca Scientifica 1, 00133
Roma, Italy}

\textit{E-mail address:} fbracci@mat.uniroma2.it

\bigskip

{\large Marina Levenshtein, Department of Mathematics, ORT Braude
College, 21982 Karmiel, Israel}

\textit{E-mail address:} marlev@braude.ac.il

\bigskip

{\large Simeon Reich, Department of Mathematics, The Technion -
Israel Institute of Technology, 32000 Haifa, Israel}

\textit{E-mail address:} sreich@tx.technion.ac.il

\bigskip

{\large David Shoikhet, Department of Mathematics, ORT Braude
College, 21982 Karmiel, Israel}

\textit{E-mail address:} davs@braude.ac.il

\end{document}